\documentclass[10pt,reqno]{amsart}
\usepackage{amssymb}
\usepackage{amsmath, amssymb, amsthm}
\usepackage{mathrsfs}
\numberwithin{equation}{section}
\usepackage{color}
\usepackage{hyperref}
\newtheorem{theorem}{Theorem}[section]

\newtheorem{lemma}[theorem]{Lemma}

\newtheorem{prop}{Proposition}[section]
\newcommand{\dif}{\mathrm{d}}
\newcommand{\SUM}[3]{\sum\limits_{{#1}={#2}}^{#3}}

\begin{document}
\title[Stability for relaxed CNS]{Asymptotic stability of composite waves of two viscous shocks for relaxed compressible Navier-Stokes equations}
\author{Renyong Guan and Yuxi Hu}
 \thanks{\noindent  Renyong Guan,   Department of Mathematics, China University of Mining and Technology, Beijing, 100083, P.R. China, renyguan@163.com\\
\indent  Yuxi Hu, Department of Mathematics, China University of Mining and Technology, Beijing, 100083, P.R. China, yxhu86@163.com\\
 }
\begin{abstract}
This paper investigates the time asymptotic stability of composite waves formed by two shock waves within the context of one-dimensional relaxed compressible Navier-Stokes equations. We establish the nonlinear stability  of the composite waves consisting of two viscous shocks  under the condition of having two small, independent wave strengths and the presence of small initial perturbations. Furthermore, the solutions of the relaxed system are observed to globally converge over time to those of the classical system as the relaxation parameter approaches zero. The methods are based on relative entropy, the $a$-contraction with shifts theory and fundamental energy estimates.
 \\[2em]
{\bf Keywords}: Relaxed compressible Navier-Stokes equations; asymptotic stability; relaxation limit; composite waves of two shock wave; relative entropy; energy estimates \\
{\bf AMS classification code}: 76N15, 35B35, 35Q35

\end{abstract}
\maketitle
\section{Introduction}
In this paper, we study the one-dimensional isentropic compressible Navier-Stokes equations, complemented by Maxwell's constitutive relations. The equations are formulated as follows:
\begin{align}\label{1.1}
\begin{cases}
\rho_t+(\rho u)_x=0,\\
(\rho u)_t+(\rho u^2)_x+p(\rho)_x=\Pi_x,\\
\tilde \tau  (\rho) (\Pi_t+ u \Pi_x)+\Pi=\mu u_x,
\end{cases}
\end{align}
where $(t, x)\in (0, +\infty)\times \mathbb R$. Here, $\rho$, $u$, $\Pi$ represent fluid density, velocity and stress, respectively. $\mu>0$ is the viscosity constant.
The pressure $p$ is assumed to satisfy the usual $\gamma$-law, $p(\rho)=A \rho^\gamma$ where $\gamma>1$ denotes the adiabatic index and $A$ is any positive constant.
Without loss of generality, we assume 
$A=1$ in the sequel.

The constitutive relation $\eqref{1.1}_3$, first introduced by Maxwell in \cite{MAX}, serves to delineate the connection between the stress tensor and velocity gradient for non-simple fluid. The relaxation parameter $\tilde\tau=\tilde \tau(\rho)$ describes the time lag in response of the stress tensor to velocity gradient. In fact, even for simple fluid, water for example, the {\it {time lag}} does exists but it is very small ranging from 1 ps to 1 ns, see \cite{GM, FS}. However, Pelton et al. \cite{MP} showed that such a {\it time lag} cannot be neglected, even for simple fluids, in the experiments of high-frequency (20 GHz) vibration of nano-scale mechanical devices immersed in water-glycerol mixtures. It was shown that, see also \cite{DJE}, equation $\eqref{1.1}_3$ provides a general formalism with which to characterize the fluid-structure interaction of nano-scale mechanical devies vibrating in simple fluids.

Assuming that $\tilde \tau(\rho)=\tau \rho$, where $\tau$ is a positive constant, equation $\eqref{1.1}_3$ reduces to
\begin{align}\label{1.2}
\tau  \rho (\Pi_t+ u \Pi_x)+\Pi=\mu u_x.
\end{align}
From the standpoint of conservation laws, Freist\"uhler initially proposed the constitutive equation \eqref{1.2} in \cite{FRE1, FRE2} in multi-dimensional case. Under this assumption, equation \eqref{1.2} exhibits conservation properties by utilizing mass equation $\eqref{1.1}_1$, making it straightforward to define weak solutions. For ease of analysis, we restate the system \eqref{1.1} with the assumption $\tilde \tau(\rho)=\tau \rho$ in Lagrangian coordinates as follows:
\begin{equation}\label{1.3}
\begin{cases}
v_t-u_x=0,\\
u_t+p_x=\Pi_x,\\
\tau \Pi_t+v\Pi=\mu u_x,
\end{cases}
\end{equation}
where $v=\frac{1}{\rho}$ denotes the specific volume per unit mass.

We are interested in the Cauchy problem to system \eqref{1.3} for the functions
\begin{align*}
(v, u, \Pi): [0, +\infty) \times \mathbb R \rightarrow (0, \infty)\times \mathbb R \times \mathbb R
\end{align*}
with initial conditions
\begin{align} \label{1.4}
(v, u, \Pi)(0,x)=(v_0, u_0, \Pi_0)(x)\rightarrow(v_{\pm},u_{\pm},0) \quad (x\rightarrow\pm\infty).
\end{align}

The large-time behavior of solutions to system \eqref{1.3}-\eqref{1.4} is closely
related to the Riemann problem of the associated $p$-system
\begin{equation}\label{1.5}
\begin{cases}
v_t-u_x=0,\\
u_t+p(v)_x=0,
\end{cases}
\end{equation}
with the Riemann initial data:
\begin{equation}\label{1.6}
(v,u)(t=0,x)=
\begin{cases}
(v_-,u_-),\quad x<0,\\
(v_+,u_+),\quad x>0.
\end{cases}
\end{equation}

We recall that there exists a unique intermediate state $(v_m, u_m)$ connected to $(v_- ,u_-)$ and $(v_+, u_+)$ by 1-shock wave and 2-shock wave, respectively. And it satisfies Rankine-Hugoniot condition
\begin{equation}\label{3.15-1}
\begin{aligned}
\begin{cases}
\sigma_1(v_m-v_-)=-(u_m-u_-),\\
\sigma_1(u_m-u_-)=p(v_m)-p(v_-),
\end{cases}
\begin{cases}
\sigma_2(v_+-v_m)=-(u_+-u_m),\\
\sigma_2(u_+-u_m)=p(v_+)-p(v_m),
\end{cases}
\end{aligned}
\end{equation}
and Lax entropy condition
\begin{equation}\label{3.15-2}
\begin{aligned}
-\sqrt{-p^\prime(v_-)}<\sigma_1<-\sqrt{-p^\prime(v_m)},\\
\sqrt{-p^\prime(v_m)}<\sigma_2<\sqrt{-p^\prime(v_+)}.
\end{aligned}
\end{equation}
Then, the Riemann problem of the associated $p$-system has a composite wave solution $(\bar{v}, \bar{u}):=(v_1^s, u_1^s)+(v_2^s, u_2^s)-(v_m, u_m)$, where 1-shock wave solution $(v_1^s, u_1^s)$ and 2-shock wave solution $(v_2^s, u_2^s)$ defined as follows (see \cite{SMO}):
\[
(v_1^s, u_1^s)(t, x)=
\begin{cases}
(v_- ,u_-),\quad x<\sigma_1t,\\
(v_m ,u_m),\quad x>\sigma_1t,
\end{cases}
(v_2^s, u_2^s)(t, x)=
\begin{cases}
(v_m ,u_m),\quad x<\sigma_2t,\\
(v_+ ,u_+),\quad x>\sigma_2t.
\end{cases}
\]

If $\tau=0$, the system \eqref{1.3} reduces to classical compressible isentropic Navier-Stokes equations:
\begin{equation}\label{1.7}
\begin{cases}
v_t-u_x=0,\\
u_t+p(v)_x=\left(\mu\frac{u_x}{v}\right)_x.
\end{cases}
\end{equation}

The asymptotic behavior of solutions for system \eqref{1.7} and its non-isentropic counterpart has been extensively explored with a variety of initial conditions, as documented in \cite{HH, SMJ, HLM, HM, WY, MNS, MNR1, MNR2}. Notably, for shock profile initial data, Matsumura and Nishihara \cite{MNS}, Goodman \cite{GD} pioneered the establishment of the stability of traveling waves with sufficiently small initial disturbances and zero mass condition, by utilizing the anti-derivative method. Furthermore, Huang and Matsumura \cite{HM} demonstrated the asymptotic stability of a composite wave consisting of two viscous shocks within the Navier-Stokes-Fourier system, given that the shocks are of small magnitude and of the same order.

In a recent advancement, Kang, Vasseur and Wang \cite{WY} overcame the challenge of reconciling the standard anti-derivative method typically employed for viscous shock stability with the energy method used for rarefaction stability. They affirmed the stability of composite waves comprising both viscous shock and rarefaction by employing the method of relative entropy and the $a$-contraction with shifts theory. These methods was initially introduced by  Bresch and Desjardins in \cite{BD} and further developed by Kang and Vasseur in \cite{KV1, KV3, KV6, KV10}, with additional insights provided in \cite{KV8, KV11, KV2, V28}. Employing similar methodologies, Han, Kang and Kim \cite{SMJ}  have recently established the uniform convergence toward a composite of two viscous shocks for system \eqref{1.7} with independently small amplitudes. The objective of the present paper is to extend the results in \cite{SMJ} to the context of the relaxed compressible Navier-Stokes equations \eqref{1.7}.


For relaxed compressible Navier-Stokes equations, the time-asymptotic stability of both single viscous shock waves and composites of two rarefaction waves has been a subject of study. When $\tilde \tau(\rho)$ is a constant, Hu-Wang \cite{ZWH} and Hu-Wang \cite{XFH} respectively established the linear stability of the viscous shock wave and nonlinear stability of rarefaction waves. With $\tilde \tau=\tau \rho$, by checking Majda's condition on the Lopatinski determinant and Zumbrun's Evans function condition \cite{MZ, MZ2, PZ},  Freist\"uhler \cite{FRE1} get the nonlinear stability of the viscous shock waves for system \eqref{1.1} with shock profile initial data. More recently, the authors \cite{HG} have shown the nonlinear stability of composite waves of viscous shock and rarefaction, where the method of relative entropy and the $a$-contraction with shifts theory were fully used.

In this paper, we studied the time asymptotic stability for composite waves of two viscous shocks to system \eqref{1.3}-\eqref{1.4}. Note that the dissipation structure of relaxed system \eqref{1.3} is much weaker than that of classical system \eqref{1.7}, thus the BD entropy used in \cite{WY, SMJ} to prove the $a$-contraction property is not available for our system and energy estimates have new challenges. We shall follow the methods in \cite{HG} to overcome such difficulties. Here are our strategy. Instead of using BD entropy, we use of the special  hyperbolic structure of the relaxed system and the relative entropy quantities with weight function and shifts to get the $L^\infty_tL^2_x$ estimates of $(v-\widetilde v, u-\widetilde u, \Pi-\widetilde \Pi)$ and weighted $L^2_tL^2_x$ estimates of $(v-\widetilde v, \Pi-\widetilde \Pi)$ (see Lemma \ref{le4.2}).  We note that, unlike that in \cite{WY,SMJ}  where the dissipation estimates of  the derivative of $(v-\widetilde v)$ were obtained for system \eqref{1.3}, we do not have such estimate. Consequently, we introduce this estimate and subsequently absorb it through a combination of high-order and dissipation estimates.

Our main theorem are stated as follows:

\begin{theorem}\label{th1}
Let the relaxation parameter $\tau$ satisfy
\begin{equation}\label{2.5}
\tau\leq \min\{\inf\limits_{z_1\in[v_m,v_-]}\frac{\mu}{2|\sigma_1^2+p^{\prime}(z_1)|},
\inf\limits_{z_2\in[v_m,v_+]}\frac{\mu}{2|\sigma_2^2+p^{\prime}(z_2)|},
1\}.
\end{equation}

For a given constant state $(v_+,u_+)\in \mathbb{R}_+\times \mathbb{R}$,
there exist constants $\delta_0,\varepsilon_0>0$ such that the following holds true.

For any constant states $(v_-,u_-)$ and $(v_m,u_m)$ satisfying \eqref{3.15-1} with
\[|v_+-v_m|+|v_m-v_-|\leq\delta_0.\]
Denote $(\widetilde{v}_i,\widetilde{u}_i,\widetilde{\Pi}_i)(x-\sigma_i t)$ the i-viscous shock solution that are the traveling wave solutions for \eqref{1.3}-\eqref{1.4} for each $i=1,2$.
Let $(v_0,u_0,\Pi_0)$ be any initial data such that
\begin{equation}\label{1.8}
\sum\limits_{\pm}\left(\|(v_0-v_{\pm},u_0-u_{\pm})\|_{L^2(\mathbb{R}_{\pm})}\right)
+\|((v_0)_x,(u_0)_x)\|_{H^1(\mathbb{R})}+\sqrt{\tau}\|\Pi_0\|_{H^2(\mathbb{R})}<\varepsilon_0,
\end{equation}
where $\mathbb{R}_+:=-\mathbb{R}_-=(0,+\infty)$.
Then, the initial value problem \eqref{1.3}-\eqref{1.4} has a unique global-in-time solution $(v,u,\Pi)\in C^1((0, +\infty)\times \mathbb R)$. Moreover, there exist an absolutely continuous shift $X_i(t)$ (defined in \eqref{2.10}) such that
\begin{equation}\label{1.9}
\begin{aligned}
&v(t,x)-\left(\widetilde{v}_1(x-\sigma_1 t-X_1(t))+\widetilde{v}_2(x-\sigma_2 t-X_2(t))-v_m\right)\in C(0,+\infty;H^2(\mathbb{R})),\\
&u(t,x)-\left(\widetilde{u}_1(x-\sigma_1 t-X_1(t))+\widetilde{u}_2(x-\sigma_2 t-X_2(t))-u_m\right)\in C(0,+\infty;H^2(\mathbb{R})),\\
&\Pi(t,x)-\left(\widetilde{\Pi}_1(x-\sigma_1 t-X_1(t))+\widetilde{\Pi}_2(x-\sigma_2 t-X_2(t))\right)\in C(0,+\infty;H^2(\mathbb{R})),
\end{aligned}
\end{equation}
and
\begin{equation}\label{hu1}
\begin{aligned}
&\sup\limits_{t\in[0,+\infty)}\left(\|v-\widetilde{v}\|^2_{H^2}
+\|u-\widetilde{u}\|^2_{H^2}+\tau\|\Pi-\widetilde{\Pi}\|^2_{H^2}\right)\\
&\qquad\qquad+\int_0^{+\infty}\left(\|\left((v-\widetilde{v})_{x}, (u-\widetilde{u})_{x}\right)\|^2_{H^1}+\|\Pi-\widetilde{\Pi}\|^2_{H^2}\right)\dif t\\
&\leq C_0\left(\|v_0-\widetilde{v}_0(\cdot)\|^2_{H^2}+\|u_0-\widetilde{u}_0(\cdot)\|^2_{H^2}
+\tau\|\Pi_0-\widetilde{\Pi}_0(\cdot)\|^2_{H^2}\right)+C_0\delta_0,
\end{aligned}
\end{equation}
where $C_0$ is a universal constant independent of  $\tau$.
In addition, as $t\rightarrow+\infty$,
\begin{equation}\label{1.10}
\begin{aligned}
\sup\limits_{x\in\mathbb{ R}}&\Big|v(t,x)-\left(\widetilde{v}_1(x-\sigma_1 t-X_1(t))+\widetilde{v}_2(x-\sigma_2 t-X_2(t))-v_m\right),\\
&u(t,x)-\left(\widetilde{u}_1(x-\sigma_1 t-X_1(t))+\widetilde{u}_2(x-\sigma_2 t-X_2(t))-u_m\right),\\
&\sqrt{\tau}\left(\Pi(t,x)-\left(\widetilde{\Pi}_1(x-\sigma_1 t-X_1(t))+\widetilde{\Pi}_2(x-\sigma_2 t-X_2(t))\right)\right)\Big|
 \rightarrow0,
\end{aligned}
\end{equation}
where
\begin{equation}\label{1.11}
\lim\limits_{t\rightarrow+\infty}|\dot{X}_i(t)|=0 \qquad for \qquad i=1,2.
\end{equation}
In addition, the shifts are well-separated in the following sense:
\begin{equation}\label{1.12}
X_1(t)+\sigma_1 t\le \frac{\sigma_1 t}{2}< 0< \frac{\sigma_2 t}{2}\le X_2(t)+\sigma_2 t,\quad t>0.
\end{equation}
\end{theorem}

Furthermore, based on the uniform estimates of error terms \eqref{hu1}, we have the following convergence theorem.
\begin{theorem}\label{th1.2}
Let $(v^{\tau}, u^{\tau}, \Pi^{\tau})$ be the global solutions obtained in Theorem \ref{th1}. Then, there exists functions $(v^0 ,u^0)\in L^{\infty}\left((0, +\infty);H^2\right)$ and $\Pi^0\in L^2\left((0, +\infty);H^2\right)$, such that, as $\tau\rightarrow0$
\begin{align*}
(v^{\tau}, u^{\tau})\rightharpoonup (v^0, u^0) \qquad weak-*\quad in \quad L^{\infty}\left((0, +\infty);H^2\right),\\
\Pi^{\tau}\rightharpoonup \Pi^0 \qquad weakly- \quad in \quad L^2\left((0, +\infty);H^2\right),
\end{align*}
where $(v^0, u^0)$ is the solution to the classical one-dimensional isentropic compressible Navier-Stokes equations \eqref{1.7}, with initial value $(v_0, u_0)$. Moreover,
\[
\Pi^0=\mu\frac{(u^0)_{x}}{v^0}.
\]
\end{theorem}

The structure of this paper is as follows. Some basic concept, including viscous shock wave and $a$-contraction with shifts theory are given in Section 2. In Section 3, we reformulate the original problem and present  the a priori estimates (Proposition \ref{p1}) which gives the proof of Theorem \ref{th1} immediately. In Section 4, we give a proof  of Proposition \ref{p1}. Finally, in Section 5, we prove that the solutions of  relaxed system \eqref{1.3} converges globally in time to  that of  classical system \eqref{1.7}.

\textbf{Notations:}  $L^p(\mathbb R)$ and $W^{s,p}(\mathbb R)$  ($1\le p \le\infty$) denote the  usual Lebesgue  and Sobolev spaces over $\mathbb R$ with the norm $\|\cdot \|_{L^p}$ and $\|\cdot\|_{W^{s,p}}$, respectively. Note that, when $s=0$, $W^{0,p}=L^p$. For $p=2$, $W^{s, 2}$ are abbreviated to $H^s$ as usual.
Let $T$ and $B$ be a positive constant and a Banach space, respectively. $C^k(0,T; B)(k \ge 0 )$ denotes the space of $B$-valued $k$-times continuously differentiable functions on $[0,T]$, and $L^p(0,T; B)$ denotes the space of $B$-valued $L^p$-functions on $[0,T]$. The corresponding space $B$-valued functions on $[0,\infty)$ are defined in an analogous manner.

\section{Preliminaries}

\subsection{Traveling wave}
In this part, we first show the existence of two traveling wave solutions for system \eqref{1.3}.
Let $\xi_{i}=x-\sigma_i t,i=1,2$, with $\sigma_1^2=\frac{p(v_-)-p(v_m)}{v_m-v_-}$ and $\sigma_2^2=\frac{p(v_m)-p(v_+)}{v_+-v_m}$ are the speed of 1-shock wave and 2-shock wave, respectively. Assume the functions $(\widetilde{u}_i,\widetilde{v}_i,\widetilde{\Pi}_i)(\xi_i)$ satisfy
\begin{equation}\label{2.1}
\begin{aligned}
(\widetilde{u}_1,\widetilde{v}_1,\widetilde{\Pi}_1)(\xi_1)\rightarrow(v_m,u_m,0),
\quad(\widetilde{u}_2,\widetilde{v}_2,\widetilde{\Pi}_2)(\xi_2)\rightarrow(v_+,u_+,0),\quad(\xi\rightarrow+\infty),\\
(\widetilde{u}_1,\widetilde{v}_1,\widetilde{\Pi}_1)(\xi_1)\rightarrow(v_-,u_-,0),\quad
(\widetilde{u}_2,\widetilde{v}_2,\widetilde{\Pi}_2)(\xi_2)\rightarrow(v_m,u_m,0)\quad(\xi\rightarrow-\infty).
\end{aligned}
\end{equation}
 Plugging the form $(\widetilde{u}_i,\widetilde{v}_i,\widetilde{\Pi}_i)(\xi_i)$ into system \eqref{1.3}, we have the following ordinary differential equations
\begin{equation}\label{2.2}
\begin{cases}
-\sigma_i(\widetilde{v}_i)_{\xi_i}-(\widetilde{u}_i)_{\xi_i}=0,\\
-\sigma_i(\widetilde{u}_i)_{\xi_i}+p(\widetilde{v}_i)_{\xi_i}=(\widetilde{\Pi}_i)_{\xi_i},\\
-\sigma_i(\tau\widetilde{\Pi}_i)_{\xi_i}+\widetilde{v}_i\widetilde{\Pi}_i=\mu(\widetilde{u}_i)_{\xi_i},
\end{cases}
\end{equation}
with the far field condition \eqref{2.1}. For $i=2$, integrating the equations $\eqref{2.2}_1$ and $\eqref{2.2}_2$ with respect to $\xi_2$, it holds
\begin{equation}\label{2.3}
\begin{cases}
\sigma_2\widetilde{v}_2+\widetilde{u}_2=\sigma_2 v_m+u_m=\sigma_2 v_++u_+,\\
\widetilde{\Pi}_2=-\sigma_2(\widetilde{u}_2-u_m)+(p(\widetilde{v}_2)-p(v_m)).
\end{cases}
\end{equation}
Substituting \eqref{2.3} and $\eqref{2.2}_1$ into $\eqref{2.2}_3$, we derive that
\begin{equation}\label{2.4}
(\widetilde{v}_2)_{\xi_2}=\frac{\widetilde{v}_2h_2(\widetilde{v}_2)}{\mu\sigma_2+\tau\sigma_2 h_2^{\prime}(\widetilde{v}_2)},
\end{equation}
where $h_2(\widetilde{v}_2)=\sigma_2^2(v_m-\widetilde{v}_2)+(p(v_m)-p(\widetilde{v}_2))$.

Similarly, for $i=1$, we have
\begin{equation}\label{2.6}
(\widetilde{v}_1)_{\xi_1}=\frac{\widetilde{v}_1h_1(\widetilde{v}_1)}{\mu\sigma_1+\tau\sigma_1 h_1^{\prime}(\widetilde{v}_1)},
\end{equation}
where $h_1(\widetilde{v}_1)=\sigma_1^2(v_--\widetilde{v}_1)+(p(v_-)-p(\widetilde{v}_1))$.

The following lemma show the existence and properties of solutions for \eqref{2.2}.
\begin{lemma}\label{le2.1}
Let \eqref{2.5} hold.
 For any states $(v_-,u_-),(v_m,u_m),(v_+,u_+)\in \mathbb{R}_+\times\mathbb{R}$ and $\sigma_1<0,\sigma_2>0$ satisfying R-H condition and Lax condition, there exists a positive constant $C$ independent of $\tau$ such that the following is true: the traveling wave solutions $(\widetilde{u}_1,\widetilde{v}_1,\widetilde{\Pi}_1)(\xi_1)$ connecting $(v_-,u_-,0)$ and $(v_m,u_m,0)$ and $(\widetilde{u}_2,\widetilde{v}_2,\widetilde{\Pi}_2)(\xi_2)$ connecting $(v_m,u_m,0)$ and $(v_+,u_+,0)$ exist uniquely and satisfy
\begin{align*}
(\widetilde{v}_1)_{\xi_1}<0,\quad
(\widetilde{v}_2)_{\xi_2}>0,\quad
(\widetilde{v}_i)_{\xi_i}\sim(\widetilde{u}_i)_{\xi_i},
\end{align*}
and
\begin{align*}
&|\widetilde{v}_i(\xi_i)-v_m|\leq C\delta_i e^{-C\delta_i|\xi_i|},
|\widetilde{u}_i(\xi_i)-u_m|\leq C\delta_i e^{-C\delta_i|\xi_i|},\quad  \forall (-1)^i\xi_i<0,\\
&|\widetilde{\Pi}_i|\leq C\delta_i^2 e^{-C\delta_i|\xi_i|},\quad \quad
|\partial_{\xi_i}(\widetilde{v}_i,\widetilde{u}_i)|\leq C\delta_i^2e^{-C\delta_i|\xi_i|},  \,   |\partial_{\xi_i} \widetilde{\Pi}_i|\le C \delta_i |(\widetilde v_i)_{\xi_i}| \quad \forall\xi_i\in\mathbb{R},\\
&|\partial^k_{\xi_i}(\widetilde{v}_i,\widetilde{u}_i,\widetilde{\Pi}_i)|\leq C\delta_i|\partial_{\xi_i}\widetilde{v}_i|, \qquad \forall\xi_i\in\mathbb{R},
\end{align*}
for $i=1,2$ and $k=2,3,4$, where $\delta_i$ denote the strength of the shock as $\delta_1:=|p(v_-)-p(v_m)|\sim|v_--v_m|\sim|u_--u_m|$ and $\delta_2:=|p(v_m)-p(v_+)|\sim|v_m-v_+|\sim|u_m-u_+|$.
\end{lemma}
\begin{proof}
We only give a proof for case $i=1$. The case $i=2$ follows in a similar way.
Firstly, we note that $\sigma^2_1=\frac{p(v_-)-p(v_m)}{v_m-v_-}$ and $h_1(\widetilde{v}_1)=\sigma_1^2(v_--\widetilde{v}_1)+(p(v_-)-p(\widetilde{v}_1))$, then we have
\[
h_1(\widetilde{v}_1)=(v_--\widetilde{v}_1)\left(P(\widetilde{v}_1)-P(v_m)\right),
\]
where $P(\widetilde{v}_1)=\frac{p(v_-)-p(\widetilde{v}_1)}{v_--\widetilde{v}_1}$.

In addition, using the entropy inequality: $-\sqrt{-p^{\prime}(v_-)}>-\sqrt{-p^{\prime}(v_m)}$, we get $v_->v_m$. So, we derive that
\begin{align}\label{new-hu-1}
0<	v_m<\widetilde{v}_1<v_-.
\end{align}
 Assuming $\delta_1$ is sufficiently small such that 
$
0 < \widetilde{v}_1 - v_m < v_- - v_m < C\delta_1<\frac{p^{\prime\prime}(v_m)}{4}.
$
and applying Taylor expansion to the function $P(\widetilde{v}_1)$ about $v_m$, we have  
\[
|P(\widetilde{v}_1)- P(v_m)- P'(v_m)(\widetilde{v}_1 - v_m)| \le C|\widetilde{v}_1 - v_m|^2 \le C\delta_1 (\widetilde{v}_1 - v_m).
\]
Next, the concavity of $p$ implies
\[
P^{\prime}(v_m)=\frac{-p^{\prime}(v_m)(v_--v_m)+p(v_-)-p(v_m)}
{(v_--v_m)^2}
=\frac{p^{\prime\prime}(v_\ast)}{2}>0, \quad v_\ast\in(v_m, v_-).
\]
 Thus, we get 
 \[
 P(\widetilde{v}_1)- P(v_m)\geq \frac{p^{\prime\prime}(v_m)}{4}(\widetilde{v}_1 - v_m)>0.
 \]
So, we derive that
 \begin{equation}\label{1225-1}
h_1(\widetilde v_1)>0.
\end{equation}
On the other hand, using \eqref{3.15-2} and \eqref{2.5}, it holds
\begin{equation}\label{1225-2}
	\sigma_1(\mu+\tau h^\prime_1(\widetilde{v}_1))<0.
\end{equation}
Therefore, combining \eqref{new-hu-1}, \eqref{1225-1}, \eqref{1225-2} and \eqref{2.6}, we conclude that 
 $(\widetilde{v}_1)_{\xi_1}<0$.

Let $f(\widetilde{v}_1)=\frac{\widetilde{v}_1h_1(\widetilde{v}_1)}{\mu\sigma_1+\tau\sigma_1 h^{\prime}_1(\widetilde{v}_1)}$. For any $v_1,v_2\in(v_m,v_-)$, using \eqref{2.5}, one has
\begin{align*}
|f(v_1)-f(v_2)|&=\Big|\frac{v_1h_1(v_1)}{\mu\sigma+\tau\sigma h_1^{\prime}(v_1)}-\frac{v_2h_1(v_2)}{\mu\sigma+\tau\sigma h_1^{\prime}(v_2)}\Big|\\
&\le \Big|\frac{v_1h_1(v_1)-v_2h_1(v_2)}{\mu\sigma+\tau\sigma h_1^{\prime}(v_1)}\Big|
+\Big|v_2h_1(v_2)\frac{\sigma\tau\left(h_1^{\prime}(v_1)-h_1^{\prime}(v_2)\right)}
{(\mu\sigma+\tau\sigma h_1^{\prime}(v_1))(\mu\sigma+\tau\sigma h_1^{\prime}(v_2))}\Big|\\
&\le C\left(|v_1-v_2|+|h_1(v_1)-h_1(v_2)|+|h_1^{\prime}(v_1)-h_1^{\prime}(v_2)|\right)\\
&\le C|v_1-v_2|.
\end{align*}
So, $f(\widetilde{v}_1)$ satisfies the usual Lipschitz condition. This together with \eqref{new-hu-1} implies that there exists a unique local solution of \eqref{2.6}

Next, for $\xi_1<0$, using \eqref{2.6}, we have
\[
(v_--\widetilde{v}_1)_{\xi_1}<\frac{2v_-p^{\prime\prime}(v_\ast)(v_--\widetilde{v}_1)(\widetilde{v}_1-v_m)}{-\mu\sigma_1}
\le C(v_--\widetilde{v}_1)(\widetilde{v}_1-v_m)\le C\delta_1(v_--\widetilde{v}_1),
\]
which gives
\[
v_--\widetilde{v}_1 \le C\delta_1e^{C\delta_1\xi_1}.
\]
On the other hand, if $\xi_1>0$, we have
\[
(v_m-\widetilde{v}_1)_{\xi_1}>\frac{v_mp^{\prime\prime}(v_\ast)(v_--\widetilde{v}_1)(\widetilde{v}_1-v_m)}{2\mu\sigma_1}
\ge C(v_--\widetilde{v}_1)(\widetilde{v}_1-v_m)\ge C\delta_1(\widetilde{v}_1-v_m).
\]
Thus, it yields
\[
\widetilde{v}_1-v_m \le C\delta_1e^{-C\delta_1\xi_1}.
\]
Combining the above estimates, one obtain
\[
(\widetilde{v}_1)_{\xi_1}\le C\delta_1^2e^{-C\delta_1|\xi_1|}.
\]
For high-order estimates, using \eqref{2.6} and \eqref{2.5}, we have
\[
|(\widetilde{v}_1)_{\xi_1\xi_1}|=\Big|
\frac{(\widetilde{v}_1)_{\xi_1}h((\widetilde{v}_1)_{\xi_1})
+\widetilde{v}_1h^{\prime}(\widetilde{v}_1)(\widetilde{v}_1)_{\xi_1}}
{\mu\sigma_1+\tau\sigma_1 h^{\prime}(\widetilde{v}_1)}
+\frac{\widetilde{v}_1 h(\widetilde{v}_1)\tau\sigma_1 h^{\prime\prime}(\widetilde{v}_1)(\widetilde{v}_1)_{\xi_1}}
{\left(\mu\sigma_1+\tau\sigma_1 h^{\prime}(\widetilde{v}_1)\right)^2}\Big|
\le C\delta_1|(\widetilde{v}_1)_{\xi_1}|.
\]
Similarly, for $k=3,4$, we can get $|\partial_{\xi_1}^k\widetilde{v}_1|\le C\delta_1|(\widetilde{v}_1)_{\xi_1}|$.

Using $\eqref{2.2}_1$, we get
\[
|\widetilde{u}_1-u_-|=|\sigma_1(\widetilde{v}_1-v_-)|\leq C\delta_1 e^{-C\delta_1|\xi_1|},\quad \xi_1<0,
\]
and
\[
|\widetilde{u}_1-u_m|=|\sigma_1(\widetilde{v}_1-v_m)|\leq C\delta_1 e^{-C\delta_1|\xi_1|},\quad \xi_1>0.
\]
Similarly,  one has
\[
|(\widetilde{u}_1)_{\xi_1}|=|\sigma_1(\widetilde{v}_1)_{\xi_1}|\le C\delta_1^2e^{-C\delta_1|\xi_1|},
\]
and
\[
|\partial_{\xi_1}^k\widetilde{u}_1|=|\sigma_1\partial_{\xi_1}^k\widetilde{v}_1|
\le C\delta_1|(\widetilde{v}_1)_{\xi_1}|,
\]
where $k=2,3,4$.

Substituting $\eqref{2.2}_1$ and $\eqref{2.2}_2$ into $\eqref{2.2}_3$, it yields
\begin{equation}\label{11.12}
\widetilde{\Pi}_1=\frac{-\sigma_1^3\tau+\sigma_1\tau p^{\prime}(\widetilde{v}_1)-\mu\sigma_1}{\widetilde{v}_1}(\widetilde{v}_1)_{\xi},
\end{equation}
and thus
\begin{equation}\label{eq29}
|\widetilde{\Pi}_1|\le C |(\widetilde{v}_1)_{\xi}|.
\end{equation}
Taking the derivative of equation \eqref{11.12} with respect to $\xi_1$, we derive that
\[
(\widetilde{\Pi}_1)_{\xi_1}=\frac{\sigma_1^3\tau+\sigma_1\tau p^{\prime}(\widetilde{v}_1)-\mu\sigma_1}{\widetilde{v}_1}(\widetilde{v}_1)_{\xi_1\xi_1}
+\frac{\sigma_1\tau p^{\prime\prime}(\widetilde{v}_1)(\widetilde{v}_1)_{\xi_1}(\widetilde{v}_1)-
(\sigma_1^3\tau+\sigma_1\tau p^{\prime}(\widetilde{v}_1)-\mu\sigma_1)(\widetilde{v}_1)_{\xi_1}}
{(\widetilde{v}_1)^2}(\widetilde{v}_1)_{\xi_1}.
\]
Noting that $(\widetilde{v}_1)_{\xi_1\xi_1}\le C\delta_1(\widetilde{v}_1)_{\xi_1}$, we have
\[
|(\widetilde{\Pi}_1)_{\xi_1}|\le C\delta_1(\widetilde{v}_1)_{\xi_1}.
\]
Similarly, for $k=2,3,4$,  we have $|\partial_{\xi}^k\widetilde{\Pi}_1|\le C\delta_1(\widetilde{v}_1)_{\xi_1}$. Thus,  the proof of this lemma is finished.
\end{proof}

Now, we turn our attention to the composite waves formed by two viscous shocks. Initially, we note that for given end states $(v_{\pm},u_{\pm},0)\in\mathbb{R}^+\times\mathbb{R}\times\mathbb{R}$ as specified in \eqref{1.4}, there exists a unique intermediate state $(v_m,u_m,0)$ such that $(v_-,u_-,0)$ is connected to $(v_m,u_m,0)$ by the 1-shock wave, and the 2-shock wave connects $(v_m,u_m,0)$ and $(v_+,u_+,0)$ for system \eqref{1.5}-\eqref{1.6}, see \cite{SMO}. The composite waves of the superposition of two viscous shocks is defined as follows:
\begin{equation}\label{2.9}
\left(\widetilde{v}_1(x-\sigma_1 t)+\widetilde{v}_2(x-\sigma_2 t)-v_m,
\widetilde{u}_1(x-\sigma_1 t)+\widetilde{u}_2(x-\sigma_2 t)-u_m,
\widetilde{\Pi}_1(x-\sigma_1 t)+\widetilde{\Pi}_2(x-\sigma_2 t)\right).
\end{equation}

\subsection{Construction of shifts}

We define the shifts $(X_1(t),X_2(t))$ as a solution to the following system of ODEs:
\begin{equation}\label{2.10}
\begin{cases}
\begin{aligned}
&\dot{X}_i(t)=-\frac{M}{\delta_i}
\left[\int_{\mathbb{R}}\frac{a}{\sigma_i}(\widetilde{u}_i^{X_i})_x(p(v)-p(\widetilde{v}))\dif x
-\int_{\mathbb{R}}a\left(p(\widetilde{v}_i^{X_i})\right)_x(v-\widetilde{v})\dif x\right],\\
&X_i(t)=0,
\end{aligned}
\end{cases}
\end{equation}
where $i=1,2$, $M$ is the specific constant chosen as $M:=\frac{5(\gamma+1)}{8\gamma p(v_m)}(-p^{\prime}(v_m))^{\frac{3}{2}}$ and $f^{X_i(t)}$ denotes a function $f$ shifted by $X_i(t)$, that is $f^{X_i(t)}(x):=f(x-X_i(t))$. The shifted weight function $a$ is defined by
\begin{align}\label{2.11}
a(t,x):=a_1(x-\sigma_1t-X_1(t))+a_2(x-\sigma_2t-X_2(t))-1,
\end{align}
where $a_1,a_2$ are two weight functions associated with 1-shock and 2-shock, respectively, defined by
\[
a_i(x-\sigma_it)=1+\frac{\lambda_i(p(v_m)-p(\widetilde{v}_i(x-\sigma_it)))}{\delta_i},
\]
where $\lambda_1,\lambda_2$ are small constants  such that $\frac{\delta_i}{\lambda_i}<\delta_0$ and $\lambda_i<C\sqrt{\delta_i}<C\delta_0$ for $i=1,2$.

On the other hand, we known that $1<a<2$ and
\[
(a_i)_x=-\frac{\lambda_i}{\delta_i}(p(\widetilde{v}_i))_x.
\]
Then, using Lemma \ref{le2.1}, we derive that
\begin{equation}\label{2.12}
\sigma_i(a_i)_x=-\sigma_i\frac{\lambda_i}{\delta_i}p^{\prime}(\widetilde{v}_i)(\widetilde{v}_i)_x>0
\end{equation}
and
\begin{equation}\label{2.13}
|(a_i)_x|\le C\frac{\lambda_i}{\delta_i}|\widetilde{v}_i|.
\end{equation}

In addition, based on the Cauchy-Lipschitz theorem, we have the following lemma, see \cite{WY, SMJ}.
\begin{lemma}\label{le3.10}
For any $c_1, c_2>0$, there exists a constant $C>0$ such that the following is true. For any $T>0$, and any function $v\in L^{\infty}((0,T )\times \mathbb{R})$ veriying
\begin{equation}\label{3.11-1}
c_1<v(t,x)<c_2,\quad \forall(t,x)\in[0, T]\times\mathbb{R},
\end{equation}
the ODE \eqref{2.10} has a unique absolutely continuous solution $(X_1, X_2)$ on $[0, T]$. Moreover,
\begin{equation}\label{3.10-2}
|X_1(t)|+|X_2(t)|\le C t, \qquad \forall \, 0\le t \le T.
\end{equation}
\end{lemma}

We remark that the shift function above is slightly different with that in \cite{SMJ} due to the unavailability of BD entropy. Nonetheless, using similar proof as in \cite{SMJ}, we can get that \eqref{2.10} has a unique absolutely continuous solution defined on any interval $[0,T]$ provided  $v(t,x)$ is bounded both above and below for any $(t,x)\in[0,T]\times \mathbb{R}$. For completeness, we give a  proof in Appendix A.

\section{Local solution and a priori estimates}

Firstly, by using the theory of symmetric hyperbolic system, we have the following local existence theory, see \cite{JR, WY, MV20, TA}.
\begin{theorem}\label{th3.1}
Let $\underline{v}$ and $\underline{u}$ be smooth monotone functions such that
\[
\underline{v}(x)=v_{\pm} \quad \underline{u}(x)=u_{\pm} \quad for \quad \pm x\geq1.
\]
For any constants $M_0,M_1,\kappa_1,\kappa_2,\kappa_3,\kappa_4$ with $M_1>M_0>0$ and $\kappa_1>\kappa_2>\kappa_3>\kappa_4>0$, there exists a constant $T_0>0$ such that if
\[
\begin{aligned}
&\|v_0-\underline{v}\|_{H^2}+\|u_0-\underline{u}\|_{H^2}+\sqrt{\tau}\|\Pi_0\|_{H^2}\leq M_0,\\
&0<\kappa_3\leq v_0(x)\leq\kappa_2,\qquad \forall x\in\mathbb{R},
\end{aligned}
\]
then the system \eqref{1.3} with initial condition \eqref{1.4} has a unique solution $(v, u, \Pi)$ on $[0,T_0]$ such that
\[
\begin{aligned}
(v-\underline{v}, u-\underline{u}, \Pi)\in C([0,T_0];H^2),
\end{aligned}
\]
and
\[
\|v-\underline{v}\|_{L^{\infty}(0,T_0;H^2)}+\|u-\underline{u}\|_{L^{\infty}(0,T_0;H^2)}
+\sqrt{\tau}\|\Pi\|_{L^{\infty}(0,T_0;H^2)}\leq M_1.
\]
Moreover:
\[
\kappa_4\leq v(t,x)\leq\kappa_1,\qquad \forall(t,x)\in[0,T_0]\times\mathbb{R}.
\]
\end{theorem}

Next, let the shifted composite waves $(\widetilde{v}, \widetilde{u}, \widetilde{\Pi})(t, x)$ as follows:
\begin{align}\label{3.1}
\left(\widetilde{v}, \widetilde{u}, \widetilde{\Pi}\right)(t, x):=
\left(\widetilde{v}_1^{X_1}+\widetilde{v}_2^{X_2}-v_m,
\widetilde{u}_1^{X_1}+\widetilde{u}_2^{X_2}-u_m,
\widetilde{\Pi}_1^{X_1}+\widetilde{\Pi}_2^{X_2}\right),
\end{align}
which satisfies the following shifted composite waves equations
\begin{equation}\label{3.2}
\begin{cases}
\widetilde{v}_t+\sum\limits_{i=1}\limits^2\dot{X}_i(\widetilde{v}_i)^{X_i}_x-\widetilde{u}_x=0,\\
\widetilde{u}_t+\sum\limits_{i=1}\limits^2\dot{X}_i(\widetilde{u}_i)^{X_i}_x+p(\widetilde{v})_x
=\widetilde{\Pi}_x+F_1,\\
\tau\widetilde{\Pi}_t+\tau\sum\limits_{i=1}\limits^2\dot{X}_i(\widetilde{\Pi}_i)^{X_i}_x+\widetilde{v}\widetilde{\Pi}
=\mu\widetilde{u}_x+F_2,
\end{cases}
\end{equation}
where $F_1,F_2$ defined as
\begin{equation}\label{3.3}
F_1:=p(\widetilde{v})_x-p(\widetilde{v}_1)^{X_1}_x-p(\widetilde{v}_2)^{X_2}_x,\quad
F_2:=\left(\widetilde{v}^{X_2}_2-v_m\right)\widetilde{\Pi}_1^{X_1}
+\left(\widetilde{v}^{X_1}_1-v_m\right)\widetilde{\Pi}_2^{X_2}.
\end{equation}
Next, we will focus on the time-asymptotic stability of the solutions to \eqref{1.3} with initial data \eqref{1.4} to the the shifted composite waves for system \eqref{3.2}.
Denote $\left(\widetilde v_0( \cdot), \widetilde u_0( \cdot), \widetilde \Pi_0( \cdot)\right)=\left(\widetilde v, \widetilde u, \widetilde \Pi\right)(t, \cdot)\Big|_{t=0}$. Then,  one can easily get
\begin{align}\label{hu4.11}
\sum\limits_{\pm}\left(\|(\widetilde v_0-v_{\pm},\widetilde u_0-u_{\pm}, \sqrt{\tau}\widetilde \Pi_0)\|_{H^2(\mathbb{R}_{\pm})}\right)
\le C (\delta_1+\delta_2).
\end{align}
The following proposition gives the a priori estimates of the error term: $(v-\widetilde{v}, u-\widetilde{u}, \Pi-\widetilde{\Pi})$.

\begin{prop}\label{p1}
Let $(v, u, \Pi)$ be local solutions given by Theorem \ref{th3.1} on $[0,T]$ for some $T>0$ and $(\widetilde{v},\widetilde{u},\widetilde{\Pi})$ be defined in \eqref{3.1}.  Then there exist positive constants $\delta_0,\varepsilon_1$ such that if two independent shock waves strength satisfy $\delta_1,\delta_2<\delta_0$ and
\begin{equation}\label{3.5}
\sup_{0\le t\le T} \|(v-\widetilde{v}, u-\widetilde{u}, \sqrt{\tau}(\Pi-\widetilde{\Pi}))\|_{H^2}\leq\varepsilon_1,
\end{equation}
then the following estimates hold
\begin{equation}\label{3.6}
\begin{aligned}
&\sup\limits_{t\in[0,T]}\left(\|v-\widetilde{v}\|^2_{H^2}
+\|u-\widetilde{u}\|^2_{H^2}+\tau\|\Pi-\widetilde{\Pi}\|^2_{H^2}\right)
+\int_0^T\sum\limits_{i=1}\limits^2\delta_i|\dot{X}_i|^2\dif t\\
&\qquad\qquad+\int_0^T\left(G^s(U)
+\|\left((v-\widetilde{v})_{x}, (u-\widetilde{u})_{x}\right)\|^2_{H^1}+\|\Pi-\widetilde{\Pi}\|^2_{H^2}\right)\dif t\\
&\leq C_0\left(\|v_0-\widetilde{v}_0(\cdot)\|^2_{H^2}+\|u_0-\widetilde{u}_0(\cdot)\|^2_{H^2}
+\tau\|\Pi_0-\widetilde{\Pi}_0(\cdot)\|^2_{H^2}\right)+C_0\delta_0.
\end{aligned}
\end{equation}
Here $C_0$ independent of $T$ and $\tau$ and
\begin{equation}\label{3.7}
G^s(U):=\sum\limits_{i=1}\limits^2\int_\mathbb{R}|(\widetilde{v}_i)^{X_i}_{x}| |\phi_i(v -\widetilde{v})|^2\dif x
\end{equation}
where cutoff functions $\phi_1,\phi_2$ are defined as follows
\begin{equation}\label{ctf9.15}
\begin{aligned}
&\phi_1:=
\begin{cases}
1,\qquad\qquad\qquad \qquad \qquad \qquad \qquad    if \quad  x<\frac{X_1(t)+\sigma_1t}{2},  \\
0,\qquad\qquad\qquad \qquad \qquad\qquad \qquad     if \quad  x>\frac{X_1(t)+\sigma_1t}{2},  \\
\text{linearly decreasing 1 to 0} \qquad\qquad   if \quad  \frac{X_1(t)+\sigma_1t}{2}\ge x\le \frac{X_1(t)+\sigma_1t}{2},
\end{cases}\\
&\phi_2(x):=1-\phi_1(x).
\end{aligned}
\end{equation}
In addition, by \eqref{2.10},
\begin{equation}\label{3.8}
|\dot{X}_1(t)|+|\dot{X}_2(t)|\leq C_0\|(v-\widetilde{v})(t,\cdot)\|_{L^{\infty}},\qquad\forall t\leq T.
\end{equation}
\end{prop}
We will give the proof of Proposition \ref{p1} in Section 4.

Now, by using Theorem \ref{th3.1} and Proposition \ref{p1}, we are able to prove Theorem \ref{th1}.

{\bf{Proof of Theorem \ref{th1}}}:   Firstly, by definition of $\underline v, \underline u$, for some $C_\ast>0$, we have
\begin{align}\label{new-hu1-1}
\sum\limits_{\pm}\left(\|(\underline{v}-v_{\pm},\underline{u}-u_{\pm})\|_{L^2(\mathbb{R}_{\pm})}\right)
+\|((\underline{v})_x,(\underline{u})_x)\|_{H^1}<C_\ast(\delta_1+\delta_2).
\end{align}
Then, using \eqref{hu4.11}, we can derive that for some $C_{\ast\ast}>0$,
\begin{equation}\label{hu4.11-2}
\begin{aligned}
&\|(\underline{v}-\widetilde{v}_0,\underline{u}-\widetilde{u}_0)\|_{H^2}
+\sqrt{\tau}\|\widetilde{\Pi}_0\|_{H^2}\\
&\le
\sum\limits_{\pm}\left(\|(\underline{v}-v_{\pm},\underline{u}-u_{\pm})\|_{H^2(\mathbb{R}_{\pm})}\right)
+\sum\limits_{\pm}\left(\|(\widetilde v_0-v_{\pm},\widetilde u_0-u_{\pm}, \sqrt{\tau}\widetilde \Pi_0)\|_{H^2(\mathbb{R}_{\pm})}\right)
\\
&<C_{\ast\ast}(\sqrt{\delta_1}+\sqrt{\delta_2}).
\end{aligned}
\end{equation}
Noting that $\delta_1, \delta_2\in(0, \delta_0)$ and using the smallness of $\delta_0$, we get
\begin{equation}\label{36.1}
\frac{\frac{\varepsilon_1}{2}-C_0\delta_0}{C_0+1}-C_{\ast\ast}(\sqrt{\delta_1}+\sqrt{\delta_2})
-C_{\ast}(\delta_1+\delta_2)>0.
\end{equation}
Then, based on the above positive constants, we define $\varepsilon_0$:
\[
\varepsilon_0:=\varepsilon_\ast-C_\ast(\delta_1+\delta_2),\qquad
\varepsilon_\ast:=\frac{\frac{\varepsilon_1}{2}-C_0\delta_0}{C_0+1}-C_{\ast\ast}(\sqrt{\delta_1}+\sqrt{\delta_2}),
\]
where we can choose $\varepsilon_0$ independent of $\delta_1, \delta_2$. For example, we can take $\varepsilon_0=\frac{\varepsilon_1}{4 (C_0+1)}$.

Now, by using the assumption \eqref{1.8} in Theorem \ref{th1} and \eqref{new-hu1-1}, we can get
\begin{equation}\label{hu4.11-3}
\begin{aligned}
&\|(v_0-\underline{v},u_0-\underline{u})\|_{H^2}
+\sqrt{\tau}\|\Pi_0\|_{H^2}\\
&\le\sum\limits_{\pm}\left(\|(v_0-v_{\pm},u_0-u_{\pm},\sqrt{\tau}\Pi_0)\|_{H^2(\mathbb{R}_{\pm})}\right)
+\sum\limits_{\pm}\left(\|(\underline{v}-v_{\pm},\underline{u}-u_{\pm})\|_{H^2(\mathbb{R}_{\pm})}\right)
\\
&<\varepsilon_0+C_{\ast}(\delta_1+\delta_2)=\varepsilon_\ast.
\end{aligned}
\end{equation}
Especially, by applying Sobolev's embedding theorem, we obtain
\[
\|v_0-\underline{v}\|_{L^\infty}\le C\varepsilon_\ast,
\]
and subsequently, using the smallness of $\varepsilon_\ast$, we can derive that
\begin{equation}\label{36.2}
\frac{v_-}{2}<v_0(x)<2v_+,\qquad \forall x\in\mathbb{R}.
\end{equation}
Applying the smallness of $\delta_1, \delta_2$, we have
\begin{equation}\label{36.3}
0<\varepsilon_\ast<\frac{\varepsilon_1}{2}.
\end{equation}
Therefore, according Theorem \ref{th3.1} with \eqref{hu4.11-3}, \eqref{36.2} and \eqref{36.3}, we obtain that the system \eqref{1.3} with initial condition \eqref{1.4} has a unique solution $(v, u, \Pi)$ on $[0,T_0]$ such that
\begin{equation}\label{36.9}
\begin{aligned}
(v-\underline{v}, u-\underline{u}, \Pi)\in C([0,T_0];H^2),
\end{aligned}
\end{equation}
\begin{equation}\label{36.7}
\|v-\underline{v}\|_{L^{\infty}(0,T_0;H^2)}+\|u-\underline{u}\|_{L^{\infty}(0,T_0;H^2)}
+\sqrt{\tau}\|\Pi\|_{L^{\infty}(0,T_0;H^2)}\leq \frac{\varepsilon_1}{2}
\end{equation}
and
\begin{equation}\label{36.5}
\frac{v_-}{3}\leq v(t,x)\leq 3v_+,\qquad \forall(t,x)\in[0,T_0]\times\mathbb{R}.
\end{equation}
In addition, using the definition of $(\widetilde{v}, \widetilde{u}, \widetilde{\Pi})$ and Lemma \ref{le2.1}, we have
\[
\begin{aligned}
&\int_0^\infty |\widetilde{v}_2(x-\sigma_2 t-X_2(t))-v_+|^2\dif x\\
&\le \int_0^\infty |\widetilde{v}_2(x-\sigma_2 t)-v_+|^2\dif x
+\int_{-X_2(t)}^0 |\widetilde{v}_2(x-\sigma_2 t)-v_+|^2\dif x\\
&\le C\delta_2(1+|X_2(t)|).
\end{aligned}
\]
Similarly, we can get
$$
\|\widetilde{v}_2-v_m\|_{L^2(\mathbb{R}_-)}^2\le C\delta_2(1+|X_2(t)|),
$$
and
$$
 \|\widetilde{v}_1-v_-\|_{L^2(\mathbb{R}_-)}^2\le C\delta_1(1+|X_1(t)|), \, \|\widetilde{v}_1-v_m\|_{L^2(\mathbb{R}_+)}^2\le C\delta_1(1+|X_1(t)|).
$$
Hence, combining the above estimates and using Lemma \ref{le2.1}, one get
\[
\begin{aligned}
&\|(\underline{v}-\widetilde{v},\underline{u}-\widetilde{u})\|_{H^2}
+\sqrt{\tau}\|\widetilde{\Pi}\|_{H^2}\\
&\le
\sum\limits_{\pm}\left(\|(\underline{v}-v_{\pm},\underline{u}-u_{\pm})\|_{H^2(\mathbb{R}_{\pm})}\right)
+\sum\limits_{\pm}\left(\|(\widetilde v-v_{\pm},\widetilde u-u_{\pm}, \sqrt{\tau}\widetilde \Pi)\|_{H^2(\mathbb{R}_{\pm})}\right)
\\
&<C\sqrt{\delta_1}(1+\sqrt{|X_1(t)|})+C\sqrt{\delta_2}(1+\sqrt{|X_2(t)|}).
\end{aligned}
\]
Then, using Lemma \ref{le3.10}, we can get
\begin{equation}\label{36.4}
\|(\underline{v}-\widetilde{v},\underline{u}-\widetilde{u})\|_{H^2}
+\sqrt{\tau}\|\widetilde{\Pi}\|_{H^2}<C\sqrt{\delta_0}(1+\sqrt{t}).
\end{equation}
Next, choosing $0<T_1<T_0$ small enough such that $C\sqrt{\delta_0}(1+\sqrt{t})<\frac{\varepsilon_1}{2}$, we have
\begin{equation}\label{36.6b}
\|(\underline{v}-\widetilde{v},\underline{u}-\widetilde{u})\|_{L^\infty(0, T_1; H^2)}
+\sqrt{\tau}\|\widetilde{\Pi}\|_{L^\infty(0, T_1; H^2)}<\frac{\varepsilon_1}{2}.
\end{equation}
Combining \eqref{36.7} and \eqref{36.6b}, we obtain that
\begin{equation}\label{36.6}
\|(v-\widetilde{v},u-\widetilde{u})\|_{L^\infty(0, T_1; H^2)}
+\sqrt{\tau}\|\Pi-\widetilde{\Pi}\|_{L^\infty(0, T_1; H^2)}<\varepsilon_1.
\end{equation}
Additionally, since the shifts $X_i(t)$ are Lipschitz continuous and using \eqref{36.6}, we can get
\begin{equation}\label{36.9}
\begin{aligned}
(v-\widetilde v, u-\widetilde u, \Pi-\widetilde \Pi)\in C([0,T_1];H^2),
\end{aligned}
\end{equation}
Now, we consider the maximal existence time defined as follows
\[
T_m:=\sup\left\{t>0\Big|\sup\limits_{[0, t]}\left(\|(v-\widetilde{v},u-\widetilde{u})\|_{H^2}
+\sqrt{\tau}\|\Pi-\widetilde{\Pi}\|_{H^2}\right)\le\varepsilon_1\right\}.
\]
If $T_m<\infty$, according the continuation argument, we can obtain that
\begin{equation}\label{36.10}
\sup\limits_{[0, T_m]}\left(\|(v-\widetilde{v},u-\widetilde{u})\|_{H^2}
+\sqrt{\tau}\|\Pi-\widetilde{\Pi}\|_{H^2}\right)=\varepsilon_1.
\end{equation}
However, by \eqref{hu4.11-2}, \eqref{36.1} and \eqref{hu4.11-3}, we can get
\begin{equation}\label{36.11}
\|(v_0-\widetilde{v}_0,u_0-\widetilde{u}_0)\|_{H^2}
+\sqrt{\tau}\|\Pi_0-\widetilde{\Pi}_0\|_{H^2}\le \frac{\frac{\varepsilon_1}{2}-C_0\delta_0}{C_0+1}.
\end{equation}
Next, based on the Proposition \ref{p1} and smallness of $\delta_0$, it can be deduced that
\begin{equation}\label{36.100}
\begin{aligned}
\sup\limits_{[0, T_m]}\left(\|(v-\widetilde{v},u-\widetilde{u})\|_{H^2}
+\sqrt{\tau}\|\Pi-\widetilde{\Pi}\|_{H^2}\right)&\le C_0\frac{\frac{\varepsilon_1}{2}-C_0\delta_0}{C_0+1}+C_0\delta_0\\
&\le \frac{\varepsilon_1}{2},
\end{aligned}
\end{equation}
which contradicts \eqref{36.10}.\\
Hence, we conclude that $T_m=\infty$ and together with Proposition \ref{p1}, we can derive that
\begin{equation}\label{x3.6}
\begin{aligned}
&\sup\limits_{t\in[0,\infty)}\left(\|v-\widetilde{v}\|^2_{H^2}
+\|u-\widetilde{u}\|^2_{H^2}+\tau\|\Pi-\widetilde{\Pi}\|^2_{H^2}\right)
+\int_0^\infty\sum\limits_{i=1}\limits^2\delta_i|\dot{X}_i|^2\dif t\\
&\qquad\qquad+\int_0^\infty\left(G^s(U)
+\|\left((v-\widetilde{v})_{x}, (u-\widetilde{u})_{x}\right)\|^2_{H^1}+\|\Pi-\widetilde{\Pi}\|^2_{H^2}\right)\dif t\\
&\leq C_0\left(\|v_0-\widetilde{v}_0(\cdot)\|^2_{H^2}+\|u_0-\widetilde{u}_0(\cdot)\|^2_{H^2}
+\tau\|\Pi_0-\widetilde{\Pi}_0(\cdot)\|^2_{H^2}\right)+C_0\delta_0.
\end{aligned}
\end{equation}

In addition, using system \eqref{4.1} below and \eqref{x3.6}, we can get
$$
\int_0^{\infty}\left\|\partial_{x}(v-\widetilde{v}, u-\widetilde{u}, \Pi-\widetilde{\Pi})\right\|_{L^2}^2\dif t \le C
$$
and
$$\int_0^{\infty} \Big|\frac{\dif}{\dif t}\left\|\partial_{x}(v-\widetilde{v}, u-\widetilde{u}, \sqrt{\tau}(\Pi-\widetilde{\Pi}))\right\|_{L^2}^2 \Big |\dif t \le C.$$
Thus, combining the interpolation inequality, we get
\[
\lim\limits_{t\rightarrow\infty}\left\|(v-\widetilde{v}, u-\widetilde{u}, \sqrt{\tau}(\Pi-\widetilde{\Pi}))\right\|_{L^{\infty}}=0.
\]
Furthermore, using the definition of shifts $X_1(t), X_2(t)$ (see \eqref{2.10}), we have
\begin{equation}\label{3.11}
|\dot{X}_1(t)|+|\dot{X}_2(t)|\leq C\left\|v-\widetilde{v}\right\|_{L^{\infty}}\rightarrow0\quad as \quad t\rightarrow\infty.
\end{equation}

On the other hand, using  \eqref{x3.6} and Sobolev's embedding theorem, we have
\[
\|v-\widetilde v\|_{L^{\infty}(\mathbb{R}_+\times \mathbb{R})}\le C_0(\varepsilon_1+\delta_0).
\]
Thus,  using \eqref{3.11} and the smallness of $\varepsilon_1$ and $\delta_0$, we get
\[
X_1(t)\le -\frac{\sigma_1}{2}t, \qquad X_2(t)\ge -\frac{\sigma_2}{2}t,\qquad t >0,
\]
or equivalently,
\begin{equation}\label{1225-3}
X_1(t)+\sigma_1t\le \frac{\sigma_1}{2}t,\qquad X_2(t)+\sigma_2t\ge \frac{\sigma_2}{2}t, \qquad t >0.
\end{equation}

Therefore,  the proof of Theorem \ref{th1} is finished.
\section{Proof of the Proposition \ref{p1}}
In this section, we establish the a priori estimates  of local solutions and thus give a proof of Proposition \ref{p1}. The constant $C$ (may different in different place) denotes a universal constant independent of $\tau$ and $T$.

Using \eqref{1.3} and \eqref{3.2}, we derive the equations for error term $(v-\widetilde{v},u-\widetilde{u},\Pi-\widetilde{\Pi})(t,x)$ as follows:
\begin{equation}\label{4.1}
\begin{cases}
(v-\widetilde{v})_t-\SUM i 1 2 \dot{X}_i \cdot (\widetilde{v}_i)^{X_i}_x-(u-\widetilde{u})_x=0,\\
(u-\widetilde{u})_t-\SUM i 1 2 \dot{X}_i \cdot (\widetilde{u}_i)^{X_i}_x+\left(p(v)-p(\widetilde{v})\right)_x
=(\Pi-\widetilde{\Pi})_x-F_1,\\
\tau(\Pi-\widetilde{\Pi})_t-\tau\SUM i 1 2\dot{X}_i \cdot (\widetilde{\Pi}_i)^{X_i}_x+v\Pi-\widetilde{v}\widetilde{\Pi}
=\mu(u-\widetilde{u})_x-F_2,
\end{cases}
\end{equation}
where $F_1,F_2$ are defined in \eqref{3.3}.

To prove Proposition \ref{p1}, we need to do  the lower-order, higher-order and dissipation estimates, respectively.

Firstly, let $U(t, x):=(v, u, \Pi)^T(t, x)$ and $\widetilde{U}(t, x):=(\widetilde{v}, \widetilde{u}, \widetilde{\Pi})^T(t, x)$ are the solution of system \eqref{1.3} and \eqref{3.2}, respectively. Then, we define a relative entropy quantity  as
\begin{equation}\label{4.2}
\eta(U|\widetilde{U})=\frac{|u-\widetilde{u}|^2}{2}+H(v|\widetilde{v})+\frac{\tau|\Pi-\widetilde{\Pi}|^2}{2\mu},
\end{equation}
where $H(v|\widetilde{v})=H(v)-H(\widetilde{v})-H^{\prime}(\widetilde{v})(v-\widetilde{v})$ and potential energy $H(v):=v^{-\gamma+1}/(\gamma-1)$. In addition, for the pressure function $p(v)=v^{-\gamma}$, denote by $p(v|\tilde v)=p(v)-p(\widetilde{v})-p^{\prime}(\widetilde{v})(v-\widetilde{v})$. Then, we have the following Lemma, see \cite{WY,SMJ}.
\begin{lemma}\label{le4.1}
For given constants $\gamma>1$, and $v_->0$, there exist constants $C,\delta_{\ast}>0$, such that the following holds true.
\begin{itemize}
\item[1)] For any $v,w$ such that $0<w<2v_-,0<v<3v_-$,
\[
|v-w|^2\leq CH(v|w),
\]
\[
|v-w|^2\leq Cp(v|w).
\]
\item[2)] For any $v,w>v_-/2$,
\[
|p(v)-p(w)|\leq C|v-w|.
\]
\item[3)] For any $0<\delta<\delta_{\ast}$, and for any $(v,w)\in \mathbb{R}^2_+$ satisfying $|p(v)-p(w)|<\delta$, and $|p(w)-p(v_-)|<\delta$, the following holds true:
\[
p(v|w)\leq\left(\frac{\gamma+1}{2\gamma}\frac{1}{p(w)}+C\delta\right)|p(v)-p(w)|^2,
\]
\[
H(v|w)\geq\frac{p(w)^{-\frac{1}{\gamma}-1}}{2\gamma}|p(v)-p(w)|^2-\frac{1+\gamma}{3\gamma^2}p(w)^{-\frac{1}{\gamma}-2}(p(v)-p(w))^3,
\]
\[
H(v|w)\leq\left(\frac{p(w)^{-\frac{1}{\gamma}-1}}{2\gamma}+C\delta\right)|p(v)-p(w)|^2.
\]

\end{itemize}
\end{lemma}

\subsection{$L^2$ energy estimates}
In the following lemmas, we get the $L^2$ estimates.
\begin{lemma}\label{le4.2}
Under the hypotheses of Proposition \ref{p1}, there exist constants $C, C_1>0$ (independent of $\tau$ and $T$) such that for all $t\in(0,T]$,
\begin{equation}\label{26.1}
\begin{aligned}
&\int_\mathbb{R}\eta(U(t,x)|\widetilde{U}(t,x))\dif x+\sum\limits_{i=1}\limits^2\frac{\delta_i}{4M}\int_0^t|\dot{X}_i|^2\dif t+C_1(1-C(\delta_0+\varepsilon_1))\int_0^tG^s(U)\dif t\\
&\quad+(1-\kappa-C(\delta^{\frac{1}{2}}_1+\delta^{\frac{1}{2}}_2))\int_0^tG(U)\dif t\\
 &\leq  C\int_\mathbb{R}\eta(U(0,x)|\widetilde{U}(0,x))\dif x
+\frac{3\mu}{4v_m}\frac{1+\kappa+C\varepsilon_1}{1-C\varepsilon_1}\int_0^t\int_\mathbb{R}|p^\prime(v)||\partial_x(v-\widetilde{v})|^2\dif x\dif t+C\delta_0,
\end{aligned}
\end{equation}
where $\kappa$ is small constant to be determined later,
\begin{align}\label{26.1b}
 G(U):=\int_\mathbb{R}\frac{v}{\mu}|\Pi-\widetilde{\Pi}|^2\dif x,
\end{align}
and $G^{s}(U)$ is given in \eqref{3.7}.
\end{lemma}
Before we prove Lemma \ref{le4.2}, we need some auxiliary lemmas. Firstly, using Lemma \ref{le2.1}, we have the following estimates.

Next, using the a priori assumption \eqref{3.5}, Sobolev's embedding theorem and  \eqref{2.10}, we get  that
\[
|\dot X_i(t)|\le C\|v-\widetilde v\|_{L^{\infty}(\mathbb{R}_+\times \mathbb{R})}\le C \varepsilon_1.
\]
Thus,  for sufficiently small $\varepsilon_1$, we can derive that
\begin{equation}\label{3.10-1}
X_1(t)+\sigma_1t\le \frac{\sigma_1}{2}t,\qquad X_2(t)+\sigma_2t\ge \frac{\sigma_2}{2}t, \qquad t >0.
\end{equation}
\begin{lemma}\label{le4.3}
Assume \eqref{3.10-1}. Given $v_+>0$, there exist positive constants $\delta_0,C$ such that for any $\delta_1,\delta_2\in(0,\delta_0)$, the following estimates hold. For each $i=1,2$,
\begin{align*}
&|(\widetilde{v}_i)_x^{X_i}||\widetilde{v}-\widetilde{v}_i^{X_i}|\le C\delta_i\delta_1\delta_2e^{-C\min\{\delta_1,\delta_2\}t},\quad t>0, \quad x\in \mathbb{R},\\
&\int_\mathbb{R}|(\widetilde{v}_i)_x^{X_i}||\widetilde{v}-\widetilde{v}_i^{X_i}|\dif x\le C\delta_1\delta_2e^{-C\min\{\delta_1,\delta_2\}t},\quad t>0,\\
&\int_\mathbb{R}|(\widetilde{v}_1)_x^{X_1}||(\widetilde{v}_2)_x^{X_2}|\dif x\le
C\delta_1\delta_2e^{-C\min\{\delta_1,\delta_2\}t},\quad t>0.
\end{align*}
\end{lemma}

The proof of the above lemma which is based on Lemma \ref{le2.1} is given in \cite{SMJ}. For completeness, we give a proof of this lemma in the Appendix B.

\begin{lemma}\label{ctf9.151}
Assume \eqref{3.10-1}. Let $\phi_i$ be the functions defined in \eqref{ctf9.15}. Given $v_+>0$, there exist positive constants $\delta_0,C$ such that for any $\delta_1,\delta_2\in(0,\delta_0)$, the following estimates hold:
\begin{align*}
&\phi_2|(\widetilde{v}_1)_x^{X_1}|\le C\delta_1^2e^{-C\delta_1t},\quad
\phi_1|(\widetilde{v}_2)_x^{X_2}|\le C\delta_2^2e^{-C\delta_2t}
\quad t>0, \quad x\in \mathbb{R},\\
&\int_\mathbb{R}\phi_2|(\widetilde{v}_1)_x^{X_1}|\dif x\le C\delta_1e^{-C\delta_1t},\quad
\int_\mathbb{R}\phi_1|(\widetilde{v}_2)_x^{X_2}|\dif x\le C\delta_2e^{-C\delta_2t}
\quad t>0.
\end{align*}
\end{lemma}
By Lemma \ref{le2.1} and the definition of $\phi_i$, this lemma follows immediately. we omit the proof and refer to Appendix B of \cite{SMJ} for details.

Then, following the Lemma 4.3 in \cite{HG} (see also Lemma 4.3 in \cite{WY}, Lemma 5.2 in \cite{KV9} and Lemma 4 in \cite{V27}), we get the estimates of the relative entropy weighted by $a(t,x)$ with shift $X_i$.
\begin{lemma}\label{le4.4}
Let $a(t,x)$ be the weight function defined in \eqref{2.11} and $X_1,X_2$ be any Lipschitz continuous function. Let $U$ and $\widetilde{U}$ are the solution of system \eqref{1.3} and \eqref{3.2}. Then, we have
\begin{equation}\label{4.3}
\frac{\dif}{\dif t}\int_\mathbb{R}a(t,x)\eta(U(t,x)|\widetilde{U}(t,x))\dif x
=\sum\limits_{i=1}\limits^2\dot{X}_i(t)Y_i(U)+J^{bad}(U)-J^{good}(U),
\end{equation}
where
\begin{align*}
Y_i(U):=&-\int_\mathbb{R}(a_i)_x^{X_i}\eta(U|\widetilde{U})\dif x+
\int_\mathbb{R}a(\widetilde{u}_i)^{X_i}_x(u-\widetilde{u})\dif x
    -\int_\mathbb{R}ap^{\prime}(\widetilde{v})(\widetilde{v}_i)^{X_i}_x(v-\widetilde{v})\dif x\\
   & +\int_\mathbb{R}a\frac{\tau}{\mu}(\widetilde{\Pi}_i)^{X_i}_x(\Pi-\widetilde{\Pi})\dif x,
\end{align*}
\begin{align*}
   J^{bad}(U):=&\sum\limits_{i=1}\limits^2
   \left(\frac{1}{2\sigma_i}\int_\mathbb{R}(a_i)_{x}^{X_i}(p(v)-p(\widetilde{v}))^2\dif x
   +\sigma_i\int_\mathbb{R}a(\widetilde{v}_i)^{X_i}_{x}p(v|\widetilde{v})\dif x\right) \\
 &+\sum\limits_{i=1}\limits^2
   \left(\frac{1}{2\sigma_i}\int_\mathbb{R}(a_i)_{x}^{X_i}(\Pi-\widetilde{\Pi})^2\dif x
   -\frac{1}{2\sigma_i}\int_\mathbb{R}(a_i)^{X_i}_{x}(p(v)-p(\widetilde{v}))(\Pi-\widetilde{\Pi})\dif x\right)\\
      &-\int_\mathbb{R}a\frac{\widetilde{\Pi}}{\mu}(\Pi-\widetilde{\Pi})(v-\widetilde{v})\dif x
    -\int_\mathbb{R}a(u-\widetilde{u})F_1\dif x
    -\int_\mathbb{R}a\frac{(\Pi-\widetilde{\Pi})}{\mu}F_2\dif x,
\end{align*}
\begin{align*}
J^{good}(U):=&\sum\limits_{i=1}\limits^2
   \left(\frac{\sigma_i}{2}\int_\mathbb{R}(a_i)_{x}^{X_i}
\left(u-\widetilde{u}-\frac{p(v)-p(\widetilde{v})}{\sigma_i}+\frac{\Pi-\widetilde{\Pi}}{\sigma_i}\right)^2\dif x
   +\sigma_i\int_\mathbb{R}(a_i)^{X_i}_{x}H(v|\widetilde{v})\dif x\right)\\
&+\sum\limits_{i=1}\limits^2
   \left(\sigma_i\int_\mathbb{R}(a_i)^{X_i}_{x}\frac{\tau}{2\mu}(\Pi-\widetilde{\Pi})^2\dif x\right)
     +\int_\mathbb{R}a\frac{v}{\mu}(\Pi-\widetilde{\Pi})^2\dif x.
\end{align*}
\end{lemma}

Next, for the terms $J^{bad}(U)$ and $J^{good}(U)$, we use the following notations:
\[
J^{bad}(U):=\sum\limits_{j=1}\limits^{7}B_j(U),\quad
J^{good}(U):=G(U)+\sum\limits_{j=1}\limits^{3}G_j(U),
\]
where
\[
B_1(U):=\sum\limits_{i=1}\limits^2
   \left(\frac{1}{2\sigma_i}\int_\mathbb{R}(a_i)_{x}^{X_i}(p(v)-p(\widetilde{v}))^2\dif x\right),\quad
  B_2(U):=\sum\limits_{i=1}\limits^2
   \left(\sigma_i\int_\mathbb{R}a(\widetilde{v}_i)^{X_i}_{x}p(v|\widetilde{v})\dif x\right),
\]
\[
B_3(U):=\sum\limits_{i=1}\limits^2
   \left(\frac{1}{2\sigma_i}\int_\mathbb{R}(a_i)_{x}^{X_i}(\Pi-\widetilde{\Pi})^2\dif x\right),
B_4(U):=-\sum\limits_{i=1}\limits^2
   \left(\frac{1}{2\sigma_i}\int_\mathbb{R}(a_i)^{X_i}_{x}(p(v)-p(\widetilde{v}))(\Pi-\widetilde{\Pi})\dif x\right),
\]
\[
B_5(U):=-\int_\mathbb{R}a\frac{\widetilde{\Pi}}{\mu}(\Pi-\widetilde{\Pi})(v-\widetilde{v})\dif x ,
B_6(U):=-\int_\mathbb{R}a(u-\widetilde{u})F_1\dif x,
B_7(U):=-\int_\mathbb{R}a\frac{(\Pi-\widetilde{\Pi})}{\mu}F_2\dif x,
\]
and
\[
G(U):=\int_\mathbb{R}a\frac{v}{\mu}(\Pi-\widetilde{\Pi})^2\dif x,
G_1(U):=\sum\limits_{i=1}\limits^2
   \left(\frac{\sigma_i}{2}\int_\mathbb{R}(a_i)_{x}^{X_i}
\left(u-\widetilde{u}-\frac{p(v)-p(\widetilde{v})}{\sigma_i}+\frac{\Pi-\widetilde{\Pi}}{\sigma_i}\right)^2\dif x\right),
\]
\[
G_2(U):=\sum\limits_{i=1}\limits^2\sigma_i\int_\mathbb{R}(a_i)^{X_i}_{x}H(v|\widetilde{v})\dif x,
G_3(U):=\sum\limits_{i=1}\limits^2
   \left(\sigma_i\int_\mathbb{R}(a_i)^{X_i}_{x}\frac{\tau}{2\mu}(\Pi-\widetilde{\Pi})^2\dif x\right).
\]
For each $Y_i(U)$, we have from \eqref{4.2} that
\begin{align*}
Y_i(U):=&-\int_\mathbb{R}(a_i)_x^{X_i}\left(\frac{|u-\widetilde{u}|^2}{2}+H(v|\widetilde{v})
+\frac{\tau|\Pi-\widetilde{\Pi}|^2}{2\mu}\right)\dif x+
\int_\mathbb{R}a(\widetilde{u}_i)^{X_i}_x(u-\widetilde{u})\dif x\\
    &-\int_\mathbb{R}ap^{\prime}(\widetilde{v})(\widetilde{v}_i)^{X_i}_x(v-\widetilde{v})\dif x
    +\int_\mathbb{R}a\frac{\tau}{\mu}(\widetilde{\Pi}_i)^{X_i}_x(\Pi-\widetilde{\Pi})\dif x.
\end{align*}
We rewrite the function $Y_i$ as follows:
\[
Y_i:=\sum\limits_{j=1}\limits^{8}Y_{ij},
\]
where
\[
Y_{i1}:=\int_\mathbb{R}\frac{a}{\sigma_i}(\widetilde{u}_i)^{X_i}_x(p(v)-p(\widetilde{v}))\dif x,\quad
Y_{i2}:=-\int_\mathbb{R}ap^{\prime}(\widetilde{v}_i^{X_i})(\widetilde{v}_i)^{X_i}_x(v-\widetilde{v})\dif x,
\]
\[
Y_{i3}:=\int_\mathbb{R}a(\widetilde{u}_i)^{X_i}_x\left(u-\widetilde{u}-\frac{p(v)-p(\widetilde{v})}{\sigma_i}\right)\dif x,\quad
Y_{i4}:=-\int_\mathbb{R}a\left(p^{\prime}(\widetilde{v})-p^{\prime}(\widetilde{v}_i^{X_i})\right)(\widetilde{v}_i)^{X_i}_x(v-\widetilde{v})\dif x,
\]
\[
Y_{i5}:=\int_\mathbb{R}a\frac{\tau}{\mu}(\widetilde{\Pi}_i)^{X_i}_x(\Pi-\widetilde{\Pi})\dif x,\quad
Y_{i6}:=-\int_\mathbb{R}(a_i)_x^{X_i}\frac{\tau|\Pi-\widetilde{\Pi}|^2}{2\mu}\dif x,
\]
\[
Y_{i7}:=-\frac{1}{2}\int_\mathbb{R}(a_i)_x^{X_i}\left(u-\widetilde{u}-\frac{p(v)-p(\widetilde{v})}{\sigma_i}\right)
\cdot\left(u-\widetilde{u}+\frac{p(v)-p(\widetilde{v})}{\sigma_i}\right)\dif x,
\]
\[
Y_{i8}:=-\int_\mathbb{R}(a_i)_x^{X_i}H(v|\widetilde{v})\dif x
-\frac{1}{2\sigma_i^2}\int_\mathbb{R}(a_i)_x^{X_i}(p(v)-p(\widetilde{v}))^2\dif x.
\]
Notice from \eqref{2.10} that
\[
\dot{X}_i=-\frac{M}{\delta_i}(Y_{i1}+Y_{i2}),
\]
which implies
\begin{equation}\label{4.4}
\dot{X}_i(t)Y_i(U)=\frac{\delta_i}{M}|\dot{X}_i|^2+\dot{X}_i\sum\limits_{j=3}\limits^8Y_{ij}.
\end{equation}

Next, we establish the following lemma. Our proof is inspired by the core idea of using cutoff functions, as introduced in \cite{SMJ}, to manage perturbations. Nonetheless, we have developed a refined set of estimates and a modified technical approach to suit our specific framework. The following self-contained proof is provided to clearly delineate our methodology and its departures from the source.

\begin{lemma}\label{le4.5}
There exists a constant $C>0$ (independent of $\tau$) such that
\begin{align*}
&-\sum\limits_{i=1}\limits^2\frac{\delta_i}{2M}|\dot{X}_i|^2+B_1+B_2-G_2-\frac{3}{4}D\\
&\le C\sum\limits_{i=1}\limits^2
\int_\mathbb{R}\left(-|(\widetilde{v}_i)^{X_i}_x||p(v)-p(\widetilde{v})|^2
+|(a_i)^{X_i}_x||p(v)-p(\widetilde{v})|^3
+|(a_i)^{X_i}_x||\widetilde{v}-\widetilde{v}_i^{X_i}||p(v)-p(\widetilde{v})|^2\right)\dif x\\
&\quad+C\left(\sum\limits_{i=1}\limits^2\delta_i^2e^{-C\delta_i t}+\frac{1}{ t^2}\right)\int_\mathbb{R}\eta(U|\widetilde{U})\dif x,
\end{align*}
where
\[D=\int_\mathbb{R}\frac{a\mu}{\gamma p(v)}|\partial_x(p(v)-p(\widetilde{v}))|^2\dif x.\]
\end{lemma}

\begin{proof}
Let new variables $y_1,y_2$ as follows:
\begin{equation}\label{4.5}
y_1:=\frac{p(\widetilde{v}_1(x-\sigma_1t))-p(v_-)}{\delta_1},\qquad
y_2:=\frac{p(v_m)-p(\widetilde{v}_2(x-\sigma_2t))}{\delta_2}.
\end{equation}
For each $i$, using Lemma \ref{le2.1}, we have
\[
\frac{\dif y_1}{\dif \xi_1}=\frac{1}{\delta_1}p^{\prime}(\widetilde{v}_1)(\widetilde{v}_1)_{\xi_1}>0,\quad
\frac{\dif y_2}{\dif \xi_2}=-\frac{1}{\delta_2}p^{\prime}(\widetilde{v}_2)(\widetilde{v}_2)_{\xi_2}>0
\]
and
\[
\lim_{\xi_i\rightarrow-\infty}y_i=0,\quad\lim_{\xi_i\rightarrow+\infty}y_i=1.
\]
Then, we using new variables $y_1,y_2$ to define new perturbation $w_1,w_2$, respectively:
\begin{equation}\label{4.6}
\begin{aligned}
w_1:=&\phi_1(x+X_1(t))\Big(p(v(t,x+X_1(t)))\\
&-p\left(\widetilde{v}_1(x-\sigma_1t)+\widetilde{v}_2(x-\sigma_2t-X_2(t)+X_1(t))-v_m\right)\Big)\circ y_1^{-1},\\
w_2:=&\phi_2(x+X_2(t))\Big(p(v(t,x+X_2(t)))\\
&-p\left(\widetilde{v}_1(x-\sigma_1t-X_1(t)+X_2(t))+\widetilde{v}_2(x-\sigma_2t)-v_m\right)\Big)\circ y_2^{-1}.
\end{aligned}
\end{equation}

In addition, the following estimates are hold:
\begin{align}\label{4.7}
|\sigma_1-(-\sigma_m)|<C\delta_1,\quad |\sigma_2-\sigma_m|<C\delta_2
\end{align}
and
\begin{equation}\label{4.8}
\begin{aligned}
&|\sigma_m^2-|p^{\prime}(\widetilde{v}_i)||\le C\delta_i,\quad
\Big|\frac{1}{\sigma_m^2}-\frac{p(\widetilde{v}_i)^{-\frac{1}{\gamma}-1}}{\gamma}\Big|\le C\delta_i,
&\Big|\frac{1}{\sigma_m^2}-\frac{p(\widetilde{v})^{-\frac{1}{\gamma}-1}}{\gamma}\Big|\le C\delta_0,
\end{aligned}
\end{equation}
where $\sigma_m=\sqrt{-p^{\prime}(v_m)}$.

Firstly, we estimates the shift part $\frac{\delta_i}{2M}|\dot{X}_i|^2$. Since the estimates of $\frac{\delta_1}{2M}|\dot{X}_1|^2$ and$\frac{\delta_i}{2M}|\dot{X}_2|^2$ are similar, we only estimate $\frac{\delta_1}{2M}|\dot{X}_1|^2$.

Using \eqref{2.2}, we have
\begin{align*}
Y_{11}=\frac{1}{\sigma_1^2}\int_\mathbb{R}a\phi_1\left(p(\widetilde{v}_1)^{X_1}_x-(\widetilde{\Pi}_1)^{X_1}_x\right)
(p(v)-p(\widetilde{v}))\dif x
+\int_\mathbb{R}a\phi_2(\widetilde{v}_1)^{X_1}_x(p(v)-p(\widetilde{v}))\dif x.
\end{align*}
Changing the variable $x\rightarrow x+X_1(t)$ and applying the new variable for $y_1,w_1$, we derive that
\[
\frac{1}{\sigma_1^2}\int_\mathbb{R}a\phi_1p(\widetilde{v}_1)^{X_1}_x(p(v)-p(\widetilde{v}))\dif x
=\frac{\delta_1}{\sigma_1^2}\int_0^1a(t,x+X_1(t))w_1\dif y_1.
\]
Then, using \eqref{4.6}, \eqref{4.8} and definition of $a(t,x)$ in \eqref{2.11}, we have
\[
\Big|Y_{11}-\frac{\delta_1}{\sigma_m^2}\int_0^1w_1\dif y_1\Big|\le C\delta_1\delta_0\int_0^1|w_1|\dif y_1
+C\int_\mathbb{R}\Big(\phi_2|(\widetilde{v}_1)^{X_1}_x|+\phi_1|(\widetilde{\Pi}_1)^{X_1}_x|\Big)|p(v)-p(\widetilde{v})|\dif x.
\]
To estimates $Y_{12}$, we first notice that
\[
\Big|v-\widetilde{v}-\Big(-\frac{p(\widetilde{v})^{-\frac{1}{\gamma}-1}}{\gamma}(p(v)-p(\widetilde{v}))\Big)\Big|\le C|p(v)-p(\widetilde{v})|^2.
\]
Then, using \eqref{3.5} and \eqref{4.8}, we derive that
\[
\Big|v-\widetilde{v}-\Big(-\frac{1}{\sigma_m^2}(p(v)-p(\widetilde{v}))\Big)\Big|\le C(\delta_0+\varepsilon_1)|p(v)-p(\widetilde{v})|.
\]

Next, changing the variable $x\rightarrow x+X_1(t)$, applying the new variable for $y_1,w_1$ and using \eqref{4.8} and the definition of $a(t,x)$ in \eqref{2.11}, it holds
\[
\begin{aligned}
\Big|Y_{12}-\frac{\delta_1}{\sigma_m^2}\int_0^1w_1\dif y_1\Big|&\le
\delta_1\int_0^1a\Big|v-\widetilde v+\frac{w_1}{\sigma_m^2}\Big|\dif y_1
+\delta_1\int_0^1(a-1)\Big|\frac{w_1}{\sigma_m^2}\Big|\dif y_1\\
&\quad+C\int_\mathbb{R}\phi_2|(\widetilde{v}_1)^{X_1}_x||p(v)-p(\widetilde{v})|\dif x\\
&\le C\delta_1(\delta_0+\varepsilon_1)\int_0^1|w_1|\dif y_1
+C\int_\mathbb{R}\phi_2|(\widetilde{v}_1)^{X_1}_x||p(v)-p(\widetilde{v})|\dif x.
\end{aligned}
\]
Combining the above estimates for $Y_{11}$ and $Y_{12}$, we get
\begin{align*}
\Big|\dot{X}_1+\frac{2M}{\sigma_m^2}\int_0^1w_1\dif y_1\Big|&\le \frac{M}{\delta_1}\sum\limits_{j=1}\limits^2\Big|Y_{1j}-\frac{\delta_1}{\sigma_m^2}\int_0^1w_1\dif y_1\Big|\\
&\le C(\delta_0+\varepsilon_1)\int_0^1|w_1|\dif y_1
+\frac{C}{\delta_1}\int_\mathbb{R}\Big(\phi_2|(\widetilde{v}_1)^{X_1}_x|+\phi_1|(\widetilde{\Pi}_1)^{X_1}_x|\Big)|p(v)-p(\widetilde{v})|\dif x,
\end{align*}
which implies
\begin{align*}
\left(\Big|\frac{2M}{\sigma_m^2}\int_0^1w_1\dif y_1\Big|-|\dot{X}_1|\right)^2\le &C(\delta_0+\varepsilon_1)^2\int_0^1|w_1|^2\dif y_1
+\frac{C}{\delta_1^2}\left(\int_\mathbb{R}\phi_2|(v_1)^{X_1}_x|
|p(v)-p(\widetilde{v})|\dif x\right)^2\\
&+\frac{C}{\delta_1^2}\left(\int_\mathbb{R}\phi_1|(\widetilde{\Pi}_1)^{X_1}_x||p(v)-p(\widetilde{v})|\dif x\right)^2.
\end{align*}
On the other hand, using H\"{o}lder inequality, Lemma \ref{le2.1} and using new variable $y_1$ and $w_1$, we have
\begin{align*}
&\frac{C}{\delta_1^2}\left(\int_\mathbb{R}\phi_2|(v_1)^{X_1}_x|
|p(v)-p(\widetilde{v})|\dif x\right)^2\\
&\le\frac{C}{\delta_1^2}\int_\mathbb{R}(\phi_2|(v_1)^{X_1}_x|)^2\dif x
\int_\mathbb{R}|
p(v)-p(\widetilde{v})|^2\dif x\\
&\le \frac{C}{\delta_1^2}\|\phi_2|(v_1)^{X_1}_x|\|_{L^\infty}\int_\mathbb{R}|(v_1)^{X_1}_x|\dif x
\int_\mathbb{R}Q(v|\widetilde{v})\dif x\\
&\le C\delta_1 e^{-C\delta_1 t}\int_\mathbb{R}\eta(U|\widetilde{U})\dif x.
\end{align*}
and
\begin{align*}
&\frac{C}{\delta_1^2}\left(\int_\mathbb{R}|(\widetilde{\Pi}_1)^{X_1}_x|
|p(v)-p(\widetilde{v})|\dif x\right)^2\\
&\le\frac{C}{\delta_1^2}\int_\mathbb{R}|(\widetilde{\Pi}_1)^{X_1}_x|\dif x
\int_\mathbb{R}|(\widetilde{\Pi}_1)^{X_1}_x||
(p(v)-p(\widetilde{v}))|^2\dif x\\
&\le C\delta_1\int_\mathbb{R}|(\widetilde{v}_1)^{X_1}_x||
(p(v)-p(\widetilde{v}))|^2\dif x\\
&\le C\delta_1^2\int_0^1|w_1|^2\dif y_1.
\end{align*}
On the other hand, using the algebraic inequality, we derive that
\begin{align*}
-\frac{\delta_1}{2M}|\dot{X}_1|^2\le&-\frac{M\delta_1}{\sigma_m^4}\left(\int_0^1w_1\dif y_1\right)^2
+C\delta_1(\delta_0^2+\varepsilon_1^2)\int_0^1|w_1|^2\dif y_1
+C\delta_1^2 e^{-C\delta_1 t}\int_\mathbb{R}\eta(U|\widetilde{U})\dif x.
\end{align*}

Next, we estimates the bad term $B_1$ and good term $G_2$. Recalling that
\[
B_1(U):=\sum\limits_{i=1}\limits^2
   \underbrace{\frac{1}{2\sigma_i}\int_\mathbb{R}(a_i)_{x}^{X_i}(p(v)-p(\widetilde{v}))^2\dif x}_{=:B_{1i}},\quad
G_2(U):=\sum\limits_{i=1}\limits^2\underbrace{\sigma_i\int_\mathbb{R}(a_i)^{X_i}_{x}H(v|\widetilde{v})}
_{=:G_{2i}}\dif x.
\]
For simplicity, we only estimates the case of $i=1$. Firstly, using Lemma \ref{le4.1}, we obtain
\begin{align*}
G_{21}\geq&\underbrace{\sigma_1 \int_\mathbb{R}(a_1)^{X_1}_x\frac{p(\widetilde{v}_1^{X_1})^{-\frac{1}{\gamma}-1}}{2\gamma}
|p(v)-p(\widetilde{v})|^2\dif x}_{=:\mathcal{G}_1}
-\sigma_1 \int_\mathbb{R}(a_1)^{X_1}_x\frac{1+\gamma}{3\gamma}p(\widetilde{v})^{-\frac{1}{\gamma}-2}
(p(v)-p(\widetilde{v}))^3\dif x\\
&+\frac{\sigma_1}{2\gamma}\int_\mathbb{R}(a_1)^{X_1}_x\left(p(\widetilde{v})^{-\frac{1}{\gamma}-1}-p(\widetilde{v}_1^{X_1})^{-\frac{1}{\gamma}-1}\right)
|p(v)-p(\widetilde{v})|^2\dif x.
\end{align*}
Using \eqref{4.7} and \eqref{4.8}, we have
\begin{align*}
B_{11}\le \frac{1}{2\sigma_m}\int_\mathbb{R}|(a_1)^{X_1}_x||p(v)-p(\widetilde{v})|^2\dif x
+\frac{C\delta_1}{2\sigma_m \sigma_1}\int_\mathbb{R}|(a_1)^{X_1}_x||p(v)-p(\widetilde{v})|^2\dif x
\end{align*}
and
\[
\frac{p(\widetilde{v}_1^{X_1})^{-\frac{1}{\gamma}-1}}{2\gamma}\ge \frac{1}{2\sigma_m^2}-C\delta_1,
\quad -\frac{\sigma_1}{\sigma_m}\ge 1-C\delta_1.
\]
Hence, we have
\[
\begin{aligned}
\mathcal{G}_1&\ge (\sigma_m-C \sigma_m\delta_1) \left(\frac{1}{2\sigma_m^2}-C\delta_1\right) \int_\mathbb{R}|(a_1)^{X_1}_x|
|p(v)-p(\widetilde{v})|^2\dif x\\
&\ge\frac{1}{2\sigma_m}(1-C\delta_1)\int_\mathbb{R}|(a_1)^{X_1}_x||p(v)-p(\widetilde{v})|^2\dif x.
\end{aligned}
\]
Then, we derive that
\begin{align*}
B_{11}-\mathcal{G}_1&\le C\delta_1\int_\mathbb{R}|(a_1)^{X_1}_x||p(v)-p(\widetilde{v})|^2\dif x\\
&\le C\delta_1\lambda_1\int_\mathbb{R}\frac{|p((\widetilde{v}_1)^{X_1}_x)|}{\delta_1}|\phi_1(p(v)-p(\widetilde{v}))|^2\dif x
+C\lambda_1\int_\mathbb{R}|(\widetilde{v}_1)^{X_1}_x)|\phi_2^2|p(v)-p(\widetilde{v})|^2\dif x.
\end{align*}
By change of variable $x\rightarrow x+X_1(t)$ and  use the new variables  $y_1$ and $w_1$, we get
\[
C\delta_1\lambda_1\int_\mathbb{R}\frac{|p((\widetilde{v}_1)^{X_1}_x)|}{\delta_1}|\phi_1(p(v)-p(\widetilde{v}))|^2\dif x=C\delta_1\lambda_1\int_0^1|w_1|^2\dif y_1.
\]
On the other hand,
\[
C\lambda_1\int_\mathbb{R}|(\widetilde{v}_1)^{X_1}_x)|\phi_2^2|p(v)-p(\widetilde{v})|^2\dif x
\le C\lambda_1\delta_1^2e^{-C\delta_1 t}\int_\mathbb{R}\eta(U|\widetilde{U})\dif x.
\]
Therefore, we have
\[
B_{11}-\mathcal{G}_1\le C\delta_1\lambda_1\int_0^1|w_1|^2\dif y_1+C\lambda_1\delta_1^2e^{-C\delta_1 t}\int_\mathbb{R}\eta(U|\widetilde{U})\dif x.
\]
Now, we estimates the bad term $B_2$, which can be written as
\[
B_2(U):=\sum\limits_{i=1}\limits^2
   \underbrace{\sigma_i\int_\mathbb{R}a(\widetilde{v}_i)^{X_i}_{x}p(v|\widetilde{v})\dif x}_{=:B{2i}}.
\]
Similarly, we only estimates the case of $i=1$. First, we have
\[
B_{21}=\sigma_1\int_\mathbb{R}a(\widetilde{v}_1)^{X_1}_{x}\phi_1^2p(v|\widetilde{v})\dif x
+\sigma_1\int_\mathbb{R}a(\widetilde{v}_1)^{X_1}_{x}(1-\phi_1^2)p(v|\widetilde{v})\dif x.
\]
Together with \eqref{2.11}, \eqref{4.5}, \eqref{4.7}, \eqref{4.8} and Lemma \ref{le4.1}, we get
\begin{align*}
&\sigma_1\delta_1\int_0^1a(t, x+X_1(t))\frac{\phi_1(x)^2}{p^{\prime}(\widetilde{v}_1)}
p(v|\widetilde{v})^{-X_1}\dif y_1\\
&\le |\sigma_1|\delta_1\int_0^1(1+C\delta_0)\Big|\frac{1}{p^{\prime}(\widetilde{v}_1)}\Big|
\left(\frac{\gamma+1}{2\gamma \sigma_mp(v_m)}\frac{\sigma_mp(v_m)}{ p(\widetilde{v})}+C\varepsilon_1\right)
|w_1|^2\dif y_1\\
&\le \delta_1(\sigma_m+C\delta_0)(1+C\delta_0)\left(\frac{1}{\sigma_m^2}+C\delta_0\right)
\alpha_m\sigma_m(1+C(\delta_0+\varepsilon_1))\int_0^1|w_1|^2\dif y_1\\
&\le \delta_1\alpha_m(1+C(\delta_0+\varepsilon_1))\int_0^1|w_1|^2\dif y_1,
\end{align*}
and
\begin{align*}
\sigma_1\int_\mathbb{R}a(\widetilde{v}_1)^{X_1}_{x}(1-\phi_1^2)p(v|\widetilde{v})\dif x
&=\sigma_1\int_\mathbb{R}a(\widetilde{v}_1)^{X_1}_{x}(1+\phi_1)\phi_2p(v|\widetilde{v})\dif x\\
&\le C\int_\mathbb{R}|\widetilde{v}_1)^{X_1}_{x}|\phi_2|p(v)-p(\widetilde{v})|^2\dif x\\
&\le C\delta_1^2e^{-C\delta_1 t}\int_\mathbb{R}\eta(U|\widetilde{U})\dif x.
\end{align*}
where $\alpha_m=\frac{\gamma+1}{2\gamma\sigma_m p(v_m)}$.

For $D(U)$, noting that $\phi_1+\phi_1=1$, we derive that
 $$
 D(U)=\int_{\mathbb{R}}\frac{a\mu}{\gamma p(v)}(\phi_1+\phi_2)|\partial_x(p(v)-p(\widetilde{v}))|^2\dif x\geq \sum\limits_{i=1}\limits^{2}\int_{\mathbb{R}}\frac{a\mu}{\gamma p(v)}\phi_i^2|\partial_x(p(v)-p(\widetilde{v}))|^2\dif x.
 $$
 For any $\delta_\ast\in(0,1)$ small enough, Young's inequality yields 
 \begin{align*}
 &\int_{\mathbb{R}}\frac{a\mu}{\gamma p(v)}|\partial_x(\phi_i(p(v)-p(\widetilde{v})))|^2\dif x\\&\le (1+\delta_{\ast})\int_{\mathbb{R}}\frac{a\mu}{\gamma p(v)}\phi_i^2|\partial_x(p(v)-p(\widetilde{v}))|^2\dif x+\frac{C}{\delta_{\ast}}\int_{\mathbb{R}}\frac{a\mu}{\gamma p(v)}|\partial_x\phi_i|^2|(p(v)-p(\widetilde{v}))|^2\dif x,
 \end{align*}
 then, it holds
 \begin{align*}
 -D(U)&\le -\frac{1}{1+\delta_{\ast}}\sum\limits_{i=1}\limits^{2}\int_{\mathbb{R}}\frac{a\mu}{\gamma p(v)}|\partial_x(\phi_i(p(v)-p(\widetilde{v})))|^2\dif x+\frac{C}{\delta_{\ast}}\sum\limits_{i=1}\limits^{2}\int_{\mathbb{R}}\frac{a\mu}{\gamma p(v)}|\partial_x\phi_i|^2|(p(v)-p(\widetilde{v}))|^2\dif x\\
 &=:D_1+D_2.
 \end{align*}
Following the similar method of Lemma 4.5 in \cite{WY}, we derive 
\begin{equation}\label{1225-4}
D_1\le -\sum\limits_{i=1}\limits^2\delta_i\alpha_m(1-C(\delta_0+\varepsilon_1))
\int_0^1y_i(1-y_i)|\partial_{y_i}w_i|^2\dif y_i.
\end{equation}
From \eqref{1225-3} and \eqref{ctf9.15}, the following holds for $D_2$ for any $(t, x)\in (0,T]\times\mathbb{R}$, it holds
\[
|\partial_x\phi(t,x)|\le  \frac{4}{\sigma_2-\sigma_1}\frac{1}{2}, \quad i=1,2. 
\] 
Thus,
\begin{equation}\label{1225-5}
D_2\le \frac{C}{\delta_\ast t^2}\int_\mathbb{R}\eta(U|\widetilde{U})\dif x.
\end{equation}
Therefore, combining \eqref{1225-4} and \eqref{1225-5}, we conclude that
\[
-D(U)\le-\sum\limits_{i=1}\limits^2\delta_i\alpha_m(1-C(\delta_0+\varepsilon_1))
\int_0^1y_i(1-y_i)|\partial_{y_i}w_i|^2\dif y_i
+\frac{C}{\delta_\ast t^2}\int_\mathbb{R}\eta(U|\widetilde{U})\dif x.
\]

Combining the estimates of $\frac{\delta_i}{2M}|\dot{X}_i|^2,B_1,\mathcal{G}_i,B_2$, we have
\begin{align*}
&B_1+B_2-\mathcal{G}-\frac{3}{4}D\\
&\le\sum\limits_{i=1}\limits^2\delta_i\alpha_m\left(
(1+C(\delta_0+\varepsilon_1))\int_0^1|w_i|^2\dif y_i-\frac{3}{4}(1-C_0(\delta_0+\varepsilon_1))
\int_0^1y_i(1-y_i)|\partial_{y_i}w_i|^2\dif y_i\right)\\
&\quad +C\left(\sum\limits_{i=1}\limits^2\delta_i^2e^{-C\delta_i t}+\frac{1}{\delta_\ast t^2}\right)\int_\mathbb{R}\eta(U|\widetilde{U})\dif x,
\end{align*}
where  $\mathcal{G}:=\mathcal{G}_1+\mathcal{G}_2$ and $$\mathcal{G}_i:=\sigma_i \int_\mathbb{R}(a_i)^{X_i}_x\frac{p(\widetilde{v}_i^{X_i})^{-\frac{1}{\gamma}-1}}{2\gamma}
|p(v)-p(\widetilde{v})|^2\dif x.$$

In the following , we shall use  the {\it {Poincar\'{e}-type inequality:}}(see \cite{WY})
$$
\int_0^1\Big|f-\int_0^1f\dif y\Big|^2\dif y\le \int_0^1\frac{1}{2}y(1-y)|f^{\prime}|^2\dif y,
$$
where $f:[0,1]\rightarrow \mathbb{R}$ with $\int_0^1y(1-y)|f^{\prime}|^2\dif y<\infty$.
Noting that
\[
\int_0^1|w-\bar{w}|^2\dif y=\int_0^1w^2\dif y-\bar{w}^2,\quad \bar{w}:=\int_0^1 w \dif y,
\]
and using the smallness of $\delta_0,\delta_1,\delta_2,\delta_\ast,\varepsilon_1$, we derive that
\begin{align*}
&B_1+B_2-\mathcal{G}-\frac{3}{4}D\\
&\le\sum\limits_{i=1}\limits^2\left(
-\frac{\delta_i\alpha_m}{8}\int_0^1|w_i|^2\dif y_i
+\frac{5\delta_i\alpha_m}{4}\left(\int_0^1w_i\dif y_i\right)^2\right)
+C\left(\sum\limits_{i=1}\limits^2\delta_i^2e^{-C\delta_i t}+\frac{1}{ t^2}\right)\int_\mathbb{R}\eta(U|\widetilde{U})\dif x.
\end{align*}
Finally, combining the estimates of $\frac{\delta_i}{2M}|\dot{X}_i|^2$ and choosing $M=\frac{5}{4}\sigma_m^4\alpha_m$, we have
\begin{align*}
&\sum\limits_{i=1}\limits^2\frac{\delta_i}{2M}|\dot{X}_i|^2
+B_1+B_2-G_2-\frac{3}{4}D\\
&\le\sum\limits_{i=1}\limits^2\left(
-\frac{\delta_i\alpha_m}{16}\int_0^1|w_i|^2\dif y_i
-\frac{\sigma_i}{2\gamma}\int_\mathbb{R}(a_i)^{X_i}_x
\left(p(\widetilde{v})^{-\frac{1}{\gamma}-1}-p(\widetilde{v}_i^{X_i})^{-\frac{1}{\gamma}-1}
|p(v)-p(\widetilde{v})|^2\dif x\right)\right.\\
&\qquad\left.+\sigma_i \int_\mathbb{R}(a_i)^{X_i}_x\frac{1+\gamma}{3\gamma}p(\widetilde{v})^{-\frac{1}{\gamma}-2}
(p(v)-p(\widetilde{v}))^3\dif x\right)
+C\left(\sum\limits_{i=1}\limits^2\delta_i^2e^{-C\delta_i t}+\frac{1}{ t^2}\right)\int_\mathbb{R}\eta(U|\widetilde{U})\dif x
\end{align*}
which implies
\begin{align*}
&\sum\limits_{i=1}\limits^2\frac{\delta_i}{2M}|\dot{X}_i|^2
+B_1+B_2-G_2-\frac{3}{4}D\\
&\le C\sum\limits_{i=1}\limits^2
\int_\mathbb{R}\left(-|(\widetilde{v}_i)^{X_i}_x||\phi_i(p(v)-p(\widetilde{v}))|^2
+|(a_i)^{X_i}_x||p(v)-p(\widetilde{v})|^3
+|(a_i)^{X_i}_x||\widetilde{v}-\widetilde{v}_i^{X_i}||p(v)-p(\widetilde{v})|^2\right)\dif x\\
&\quad+C\left(\sum\limits_{i=1}\limits^2\delta_i^2e^{-C\delta_i t}+\frac{1}{ t^2}\right)\int_\mathbb{R}\eta(U|\widetilde{U})\dif x.
\end{align*}
Thus, the proof of this lemma is finished.
\end{proof}
Now, we are ready to prove Lemma \ref{le4.2}. Firstly, using Young's inequality, we have
\[
\sum\limits_{i=1}\limits^2\left(\dot{X}_i\sum\limits_{j=3}\limits^8Y_{ij}\right)\le
\sum\limits_{i=1}\limits^2\frac{\delta_i}{4M}|\dot{X}_i|^2
+\sum\limits_{i=1}\limits^2\frac{C}{\delta_i}\sum\limits_{j=3}\limits^8|Y_{ij}|^2.
\]
Then, we derive that
\begin{align*}
&\frac{\dif}{\dif t}\int_\mathbb{R}a\eta(U|\widetilde{U})\dif x \\ &\le C(-\mathcal{G}^s+K_1+K_2)
+C\left(\sum\limits_{i=1}\limits^2\delta_i^2e^{-C\delta_i t}+\frac{1}{\delta_\ast t^2}\right)\int_\mathbb{R}\eta(U|\widetilde{U})\dif x\\
&\quad-\sum\limits_{i=1}\limits^2\frac{\delta_i}{4M}|\dot{X}_i|^2
+\sum\limits_{i=1}\limits^2\frac{C}{\delta_i}\sum\limits_{j=3}\limits^8|Y_{ij}|^2
+\sum\limits_{i=3}\limits^7B_i-G_1-G_3-G+\frac{3}{4}D,
\end{align*}
where
\[
\mathcal{G}^s:=\sum\limits_{i=1}\limits^2
\int_\mathbb{R}|(\widetilde{v}_i)^{X_i}_x||\phi_i(p(v)-p(\widetilde{v}))|^2\dif x,\quad
K_1:=\sum\limits_{i=1}\limits^2
\int_\mathbb{R}|(a_i)^{X_i}_x||p(v)-p(\widetilde{v})|^3\dif x,
\]
\[
K_2:=\sum\limits_{i=1}\limits^2
\int_\mathbb{R}|(a_i)^{X_i}_x||\widetilde{v}-\widetilde{v}_i^{X_i}||p(v)-p(\widetilde{v})|^2\dif x.
\]

For $K_1$, firstly, we have
\[
\int_\mathbb{R}|(a_i)^{X_i}_x||p(v)-p(\widetilde{v})|^3\dif x
\le \frac{C\lambda_i}{\delta_i}\int_\mathbb{R}|(\widetilde v)^{X_i}_x|(\phi_i+1-\phi_i)|p(v)-p(\widetilde{v})|^3\dif x,
\]
then, using the interpolation inequality, H\"{o}lder inequality, Young's inequality,  \eqref{2.13}, Lemma \ref{le2.1} and $\lambda_i\le C\sqrt{\delta_i}$, we get
\begin{align*}
&\frac{\lambda_i}{\delta_i}\int_\mathbb{R}|(\widetilde v)^{X_i}_x|\phi_i|p(v)-p(\widetilde{v})|^3\dif x\\
&\le \frac{C\lambda_i}{\delta_i}\|p(v)-p(\widetilde{v})\|^2_{L^{\infty}}
\sqrt{\int_\mathbb{R}|(\widetilde{v}_i)^{X_i}_x||\phi_i(p(v)-p(\widetilde{v}))|^2\dif x}
\sqrt{\int_\mathbb{R}|(\widetilde{v}_i)^{X_i}_x|\dif x}\\
&\le \frac{C\lambda_i}{\sqrt{\delta_i}}\|\partial_x(p(v)-p(\widetilde{v}))\|_{L^{2}}
\|p(v)-p(\widetilde{v})\|_{L^{2}}
\sqrt{\int_\mathbb{R}|(\widetilde{v}_i)^{X_i}_x||\phi_i(p(v)-p(\widetilde{v}))|^2\dif x}\\
&\le C\varepsilon_1(D+\mathcal{G}^s).
\end{align*}
On the other hand,
\[
\frac{\lambda_i}{\delta_i}\int_\mathbb{R}|(\widetilde v)^{X_i}_x|(1-\phi_i)|p(v)-p(\widetilde{v})|^3\dif x
\le C\varepsilon_1\lambda_i\delta_ie^{-C\delta_i t}\int_\mathbb{R}\eta(U|\widetilde{U})\dif x.
\]
Thus, it holds
\[
K_1\le C\varepsilon_1(D+\mathcal{G}^s)+C\sum\limits_{i=1}\limits^2\varepsilon_1\lambda_i\delta_ie^{-C\delta_i t}\int_\mathbb{R}\eta(U|\widetilde{U})\dif x.
\]
Similarly, for $K_2$, by addtionally using Lemma \ref{le4.3}, one obtain
\begin{align*}
K_2&\le\sum\limits_{i=1}\limits^2\frac{C\lambda_i}{\delta_i}\|p(v)-p(\widetilde{v})\|_{L^4}^2
\left\||(\widetilde{v}_i)^{X_i}_x||\widetilde{v}-\widetilde{v}_i^{X_i}|\right\|_{L^2}\\
&\le \sum\limits_{i=1}\limits^2\frac{C\lambda_i}{\delta_i}\|p(v)-p(\widetilde{v})\|_{L^2}^{3/2}
\|\partial_x(p(v)-p(\widetilde{v}))\|_{L^2}^{1/2}
\left\||(\widetilde{v}_i)^{X_i}_x||\widetilde{v}-\widetilde{v}_i^{X_i}|\right\|_{L^2}\\
&\le C\varepsilon_1\sum\limits_{i=1}\limits^2\sqrt{\delta_i}
\|\partial_x(p(v)-p(\widetilde{v}))\|_{L^2}^{1/2}
\left\|\sqrt{|(\widetilde{v}_i)^{X_i}_x|}|\widetilde{v}-\widetilde{v}_i^{X_i}|\right\|_{L^2}\\
&\le C\varepsilon_1\left(D
+\sum\limits_{i=1}\limits^2\delta_i\left\|\sqrt{|(\widetilde{v}_i)^{X_i}_x|}|\widetilde{v}-\widetilde{v}_i^{X_i}|\right\|_{L^2}^2\right).
\end{align*}
For $Y_{ij}$, using \eqref{2.13}, Young's inequality and H\"{o}lder inequality, we first get
\[
\begin{aligned}
Y_{i3}&\le C\frac{\delta_i}{\lambda_i}\int_\mathbb{R}|(a_i)^{X_i}_x|
\Big|u-\widetilde{u}-\frac{p(v)-p(\widetilde{v})}{\sigma_i}\Big|\dif x\\
&\le C\frac{\delta_i}{\lambda_i}\left(\int_\mathbb{R}|(a_i)^{X_i}_x|
\dif x\right)^{\frac{1}{2}}\left(\int_\mathbb{R}|(a_i)^{X_i}_x|
\Big|u-\widetilde{u}-\frac{p(v)-p(\widetilde{v})}{\sigma_i}\Big|^2\dif x\right)^{\frac{1}{2}}\\
&\le C\frac{\delta_i}{\sqrt{\lambda_i}}\left(\int_\mathbb{R}|(a_i)^{X_i}_x|
\Big|u-\widetilde{u}-\frac{p(v)-p(\widetilde{v})}{\sigma_i}+\frac{\Pi-\widetilde{\Pi}}{\sigma_i}\Big|^2\dif x+
\int_\mathbb{R}|(a_i)^{X_i}_x|
\Big|\frac{\Pi-\widetilde{\Pi}}{\sigma_i}\Big|^2\dif x\right)^{\frac{1}{2}}\\
&\le C\frac{\delta_i}{\sqrt{\lambda_i}}\sqrt{G_1+G}.
\end{aligned}
\]
Meanwhile, using Lemmas \ref{le2.1}, \ref{le4.3} and H\"{o}lder inequality, we have
\begin{align*}
Y_{i4}&\le C\int_\mathbb{R}|\widetilde{v}-\widetilde{v}_i^{X_i}||(\widetilde{v}_i)^{X_i}_x|
|p(v)-p(\widetilde{v})|\dif x\\
&\le C\sqrt{\int_\mathbb{R}(|\widetilde{v}-\widetilde{v}_i^{X_i}||(\widetilde{v}_i)^{X_i}_x|)^2
\dif x}\sqrt{\int_\mathbb{R}\eta(U|\widetilde U)\dif x}\\
&\le C\sqrt{\delta_i}\delta_1\delta_2e^{-C\min\{\delta_1, \delta_2\}t}\sqrt{\int_\mathbb{R}\eta(U|\widetilde U)\dif x}.
\end{align*}
Similarly, for $Y_{i5},Y_{i6},Y_{i7}$, we have
\[
Y_{i5}\le C\tau\int_\mathbb{R}|(\widetilde{\Pi}_i)^{X_i}_x||\Pi-\widetilde{\Pi}|\dif x
\le C\tau\delta_i \sqrt{G},
\]
\[
Y_{i6}\le\sqrt{\tau}\|\Pi-\widetilde{\Pi}\|_{L^\infty}
\int_\mathbb{R}|(a_i)_x^{X_i}|\frac{\sqrt{\tau}|\Pi-\widetilde{\Pi}|}{2\mu}\dif x
\le C\varepsilon_1\lambda_i\sqrt{\tau\delta_i}\sqrt{G},
\]
and
\[
Y_{i7}\le C\varepsilon_1\|(a_i)_x^{X_i}\|_{L^\infty}^{\frac{1}{2}}\sqrt{G_1+G}
\le C\varepsilon_1\sqrt{\lambda_i\delta_i}\sqrt{G_1+G}.
\]
For $Y_{i8}$, by using \eqref{2.13}, Lemmas \ref{le2.1}, \ref{le4.3} and H\"{o}lder inequality, we get
\begin{align*}
\frac{C}{\delta_i}|Y_{i8}|^2
&\le\frac{C\lambda_i^2}{\delta_i^3}
\left(\int_\mathbb{R}|(\widetilde{v}_i)_x^{X_i}||p(v)-p(\widetilde{v})|^2\dif x\right)^2 \\
&\le\frac{C\lambda_i^2}{\delta_i}
\int_\mathbb{R}|p(v)-p(\widetilde{v})|^2\dif x
\int_\mathbb{R}|(\widetilde{v}_i)_x^{X_i}||p(v)-p(\widetilde{v})|^2\dif x\\
&\le C\varepsilon_1^2 \left(\int_\mathbb{R}|(\widetilde{v}_i)_x^{X_i}||\phi_i(p(v)-p(\widetilde{v}))|^2\dif x+\int_\mathbb{R}|(\widetilde{v}_i)_x^{X_i}|(1-\phi_i^2)|p(v)-p(\widetilde{v})|^2\dif x\right)\\
&\le C\varepsilon_1^2\left(\mathcal{G}^s+\delta_i^2e^{-C\delta_it}\int_\mathbb{R}\eta(U|\widetilde U)\dif x\right).
\end{align*}
Likewise, for $D$, we have
\begin{align*}
D&\le (1+\kappa)\int_\mathbb{R}\frac{\mu(p^\prime(v))^2}{\gamma p(v)}|\partial_x(v-\widetilde{v})|^2\dif x+
C\sum\limits_{i=1}\limits^2\int_\mathbb{R}|(\widetilde{v}_i)_x^{X_i}|^2|p(v)-p(\widetilde{v})|^2\dif x\\
&\le (1+\kappa)\int_\mathbb{R}\frac{\mu|p^\prime(v)|}{v_m-C\varepsilon_1}|\partial_x(v-\widetilde{v})|^2\dif x
+C\sum\limits_{i=1}\limits^2\delta_i^2\left(\mathcal{G}^s+\delta_i^2e^{-C\delta_i t}\int_{\mathbb{R}}\eta(U|\widetilde U)\dif x\right).
\end{align*}
For $B_3$, using $\|(a_i)_{x}^{X_i}\|_{L^\infty}\le C\lambda_i\delta_i$ and Lemma \ref{le2.1}, we first have
\[
B_3\le C(\lambda_1\delta_1+\lambda_2\delta_2)G.
\]
For $B_4$, using \eqref{2.13}, Lemma \ref{le2.1} and Young's inequality, we have
\begin{align*}
B_4&\le C\sum\limits_{i=1}\limits^2
   \left(\lambda_i\int_\mathbb{R}|(\widetilde{v}_i)^{X_i}_{x}|^{\frac{1}{2}}|p(v)-p(\widetilde{v})|
   |\Pi-\widetilde{\Pi}|\dif x\right)\\
&\le C\sum\limits_{i=1}\limits^2\lambda_i
   \left(\int_\mathbb{R}|(\widetilde{v}_i)^{X_i}_{x}||p(v)-p(\widetilde{v})|^2
   \dif x+G\right)\\
&\le C\sum\limits_{i=1}\limits^2\lambda_i
   \left(\mathcal{G}^s+\delta_i^2e^{-C\delta_i t}\int_{\mathbb{R}}\eta(U|\widetilde U)\dif x+G\right).
\end{align*}
Similarly, for $B_5$, using \eqref{3.1}, $|\widetilde{\Pi}_i^{X_i}|\sim|(\widetilde{v}_i)^{X_i}_x|$ and Young's inequality, we have
\begin{align*}
B_5&\le C\int_\mathbb{R}(|(\widetilde{v}_1)^{X_1}_x|+|(\widetilde{v}_2)^{X_2}_x|)
|\Pi-\widetilde{\Pi}||p(v)-p(\widetilde{v})|\dif x\\
&\le C\delta_i\sum\limits_{i=1}\limits^2
   \left(\int_\mathbb{R}|(\widetilde{v}_i)^{X_i}_{x}||p(v)-p(\widetilde{v})|^2\dif x+G\right)\\
&\le C\sum\limits_{i=1}\limits^2\lambda_i
   \left(\mathcal{G}^s+\delta_i^2e^{-C\delta_i t}\int_{\mathbb{R}}\eta(U|\widetilde U)\dif x+G\right).
\end{align*}
For $B_6,B_7$, we note that, for $i=1,2$
\begin{equation}\label{4.9}
F_i\le C\left(|(\widetilde{v}_1)^{X_1}_x||\widetilde{v}-(\widetilde{v}_1)^{X_1}_x|
+|(\widetilde{v}_2)^{X_2}_x||\widetilde{v}-(\widetilde{v}_2)^{X_2}_x|\right).
\end{equation}
So, using Lemma \ref{le4.3}, H\"{o}lder inequality and Young's inequality, we derive that
\begin{align*}
|B_6|+|B_7|
&\le C\varepsilon_1\sum\limits_{i=1}\limits^2
\left\||(\widetilde{v}_i)^{X_i}_x||\widetilde{v}-(\widetilde{v}_i)^{X_i}_x|\right\|_{L^2}
+\kappa G+C\sum\limits_{i=1}\limits^2
\left\||(\widetilde{v}_i)^{X_i}_x||\widetilde{v}-(\widetilde{v}_i)^{X_i}_x|\right\|_{L^2}^2\\
&\le \kappa G+ C\varepsilon_1\delta_1\delta_2(\delta_1^{\frac{1}{2}}+\delta_2^{\frac{1}{2}})e^{-C\min\{\delta_1,\delta_2\}t}
+ C\delta_1^2\delta_2^2(\delta_1+\delta_2)e^{-C\min\{\delta_1,\delta_2\}t}\\
&\le \kappa G+ C\delta_1\delta_2e^{-C\min\{\delta_1,\delta_2\}t}.
\end{align*}

Noting that $\frac{1}{t^2}$ is not integrable near $t=0$, so we first to get the estimate for $t\ge 1$.  Then, for any $\delta_1, \delta_2\in(0, \delta_0)$, combining all the above estimates, we derive that
\begin{align*}
&\frac{\dif}{\dif t}\int_\mathbb{R}a\eta(U|\widetilde{U})\dif x
+\sum\limits_{i=1}\limits^2\frac{\delta_i}{4M}|\dot{X}_i|^2+\frac{1}{2}G_1+G_3
+(1-\kappa-C\delta^{\frac{1}{2}}_0)G
+C_1(1-C(\delta_0+\varepsilon_1))\mathcal{G}^s\\
&\le C\delta_1\delta_2e^{-C\min\{\delta_1,\delta_2\}t}
+C\left(\sum\limits_{i=1}\limits^2\delta_i^2\delta_ie^{-C\delta_i t}+\delta_1\delta_2e^{-C\min\{\delta_1,\delta_2\}t}+\frac{1}{t^2}\right)\int_{\mathbb{R}}\eta(U|\widetilde U)\dif x\\
&\quad+\frac{3\mu}{4v_m}(1+\kappa+C\varepsilon_1)\int_\mathbb{R}|p^\prime(v)||\partial_x(v-\widetilde{v})|^2\dif x.
\end{align*}
By using Gr\"onwall inequality, it holds
\begin{align*}
&\int_\mathbb{R}a\eta(U|\widetilde{U})\dif x
+\int_1^t\left(\sum\limits_{i=1}\limits^2\frac{\delta_i}{4M}|\dot{X}_i|^2+\frac{1}{2}G_1+G_3
+(1-\kappa-C\delta^{\frac{1}{2}}_0)G
+C_1(1-C(\delta_0+\varepsilon_1))\mathcal{G}^s\right)\dif t\\
&\le C\left(\int_{\mathbb{R}}a\eta(U|\widetilde U)\dif x\Big|_{t=1}+ \frac{\delta_1\delta_2}{\min\{\delta_1,\delta_2\}}\right)\times \exp\left(\int_1^t\Big(\sum\limits_{i=1}\limits^2\delta_i^2\delta_ie^{-C\delta_i s}+\delta_1\delta_2e^{-C\min\{\delta_1,\delta_2\}s}+\frac{1}{s^2}\Big)\dif s\right)\\
&\quad+\frac{3\mu}{4v_m}(1+\kappa)\int_1^t\int_\mathbb{R}|p^\prime(v)||\partial_x(v-\widetilde{v})|^2\dif x\dif s\\
&\le C\int_{\mathbb{R}}a\eta(U|\widetilde U)\dif x\Big|_{t=1}+\frac{3\mu}{4v_m}(1+\kappa+C\varepsilon_1)\int_1^t\int_\mathbb{R}|p^\prime(v)||\partial_x(v-\widetilde{v})|^2\dif x\dif s+C\delta_0.
\end{align*}

On the other hand, firstly, we have
\begin{align*}
\frac{\dif}{\dif t}\int_\mathbb{R}a\eta(U|\widetilde{U})\dif x
=-\sum\limits_{i=1}\limits^2\frac{\delta_i}{M}|\dot{X}_i|^2
+\sum\limits_{i=1}\limits^2\left(\dot{X}_i\sum\limits_{j=3}\limits^8Y_{ij}\right)
+\sum\limits_{i=1}\limits^7B_i-\sum\limits_{i=1}\limits^3G_i-G.
\end{align*}
Using Young's inequality, we get
 \begin{align*}
\frac{\dif}{\dif t}\int_\mathbb{R}a\eta(U|\widetilde{U})\dif x
+\sum\limits_{i=1}\limits^2\frac{\delta_i}{4M}|\dot{X}_i|^2+\sum\limits_{i=1}\limits^3G_i+G+C_1\mathcal{G}^s\le
\sum\limits_{i=1}\limits^2\left(\frac{C}{\delta_i}\sum\limits_{j=3}\limits^8|Y_{ij}|^2\right)
+\sum\limits_{i=1}\limits^7B_i+C_1\mathcal{G}^s.
\end{align*}
By the H\"{o}lder inequality, Lemma \ref{le2.1}, we have
\begin{align*}
\sum\limits_{j=3}\limits^8|Y_{ij}|&\le C\|(\widetilde{v}_i)^{X_i}_x\|_{L^2}
\left(\|p(v)-p(\widetilde{v})\|_{L^2}+\|u-\widetilde{u}\|_{L^2}+\sqrt{\tau}\|S-\widetilde{S}\|_{L^2}\right)\\
&+C\|(a_i)^{X_i}_x\|_{L^\infty}
\left(\|p(v)-p(\widetilde{v})\|_{L^2}^2+\|u-\widetilde{u}\|_{L^2}^2+\tau\|S-\widetilde{S}\|_{L^2}^2\right)\\
&\le C\delta_0\varepsilon_1,
\end{align*}
which yields
\[
\sum\limits_{i=1}\limits^2\left(\frac{C}{\delta_i}\sum\limits_{j=3}\limits^8|Y_{ij}|^2\right)
\le C\varepsilon_1^2\delta_0.
\]
Similarly, we have
\[
\sum\limits_{i=1}\limits^7B_i\le C\left(\|(a_i)^{X_i}_x\|_{L^\infty}
+\|\widetilde{S}\|_{L^\infty}\right)
\|p(v)-p(\widetilde{v})\|_{L^2}^2+C(\delta_1+\delta_2)G
\le C\varepsilon_1^2\delta_0+C\delta_0G,
\]
and
\[
\mathcal{G}^s\le C\sum\limits_{i=1}\limits^2\|(\widetilde{v}_i)^{X_i}_x\|_{L^\infty}
\|p(v)-p(\widetilde{v})\|_{L^2}^2\le C\varepsilon_1^2\delta_0.
\]
Hence, combining above estimates, we derive that
\[
\frac{\dif}{\dif t}\int_\mathbb{R}a\eta(U|\widetilde{U})\dif x
+\sum\limits_{i=1}\limits^2\frac{\delta_i}{4M}|\dot{X}_i|^2+G_1+G_3+(1-C\delta_0)G+C_1\mathcal{G}^s\le C\delta_0.
\]

 Then, for a short time $0\le t\le1$, we have
 \[
 \sup_{0\le t\le1}\int_\mathbb{R}a\eta(U|\widetilde{U})\dif x
 +\int_0^1\left(\sum\limits_{i=1}\limits^2\frac{\delta_i}{4M}|\dot{X}_i|^2+G_1+G_3+(1-C\delta_0)G
 +C_1\mathcal{G}^s\right)\dif t
 \le \int_\mathbb{R}a\eta(U|\widetilde{U})\dif x\Big|_{t=0}+C\delta_0.
 \]
Finally, integrating it over $[0,T]$ for any $t\le T$, we derive that
\begin{align*}
&\int_\mathbb{R}a(t,x)\eta(U(t,x)|\widetilde{U}(t,x))\dif x
+\int_0^t\left(\sum\limits_{i=1}\limits^2\frac{\delta_i}{4M}|\dot{X}_i|^2+\frac{1}{2}G_1+G_3
+(1-\kappa-C\delta^{\frac{1}{2}}_0)G
+C_1(1-C(\delta_0+\varepsilon_1))\mathcal{G}^s\right)\dif t\\
&\le C\int_\mathbb{R}a(0,x)\eta(U(0,x)|\widetilde{U}(0,x))\dif x
+\frac{3\mu}{4v_m}(1+\kappa)\int_0^t\int_\mathbb{R}|p^\prime(v)||\partial_x(v-\widetilde{v})|^2\dif x\dif s+C\delta_0.
\end{align*}

Using Lemma \ref{le4.1}, we have finished proof of Lemma \ref{le4.2}
\subsection{High-order energy estimates}In this section, we show the high-order energy estimates for the system \eqref{4.1}. Let $\Phi=v-\widetilde{v},\Psi=u-\widetilde{u},Q=\Pi-\widetilde{\Pi}$, then we have the following lemma.
\begin{lemma}\label{Le4.7}
Under the hypotheses of Proposition \ref{p1}, there exists $C>0$ (independent of $\tau$ and $T$) such that for all $t\in(0,T]$, we have
\begin{equation}\label{26.2}
\begin{aligned}
&\|\left(\partial_{x}\Phi, \partial_{x}\Psi\right)\|_{H^1}^2
+\tau\|\partial_{x}Q\|_{H^1}^2
+\int_0^t\|\partial_{x}Q\|_{H^1}^2\dif t\leq
C(\|\left(\partial_{x}\Phi_0, \partial_{x}\Psi_0\right)\|_{H^1}^2
+\tau\|\partial_{x}Q_0\|_{H^1}^2)\\
 &+C\delta_0\int_0^tG^s(U)\dif t
+C(\delta_0+\varepsilon_1)\int_0^t\left(\|\left(\partial_{x}\Phi, \partial_{x}\Psi\right)\|_{H^1}^2+G(U)\right)\dif t
+C\int_0^t\sum\limits_{i=1}^2\delta_i^2|\dot{X}_i|^2\dif t
+C\delta_0,
\end{aligned}
\end{equation}
where $\Phi_0=v_0-\widetilde{v}_0(\xi), \Psi_0=u_0-\widetilde{u}_0(\xi), Q_0=\Pi_0-\widetilde{\Pi}_0(\xi)$, and $G^s(U), G(U)$ are defined in \eqref{3.7} and \eqref{26.1b}, respectively.
\end{lemma}
\begin{proof}
Applying $\partial^k_x(k=1,2)$ to the system \eqref{4.1}, we derive that
\begin{equation}\label{4.10}
\begin{cases}
\partial_{t}\partial_{x}^k\Phi
-\sum\limits_{i=1}\limits^2\dot{X}_i\partial_{x}^{k+1}\widetilde{v}_i^{X_i}-\partial_{x}^{k+1}\Psi=0,\\
\partial_{t}\partial_{x}^k\Psi
-\sum\limits_{i=1}\limits^2\dot{X}_i\partial_{x}^{k+1}\widetilde{u}_i^{X_i}
          +p^{\prime}(v)\partial_{x}^{k+1}\Phi=\partial_{x}^{k+1}Q+F_3^k,\\
\tau\partial_{t}\partial_{x}^kQ
-\tau\sum\limits_{i=1}\limits^2\dot{X}_i\partial_{x}^{k+1}\widetilde{\Pi}_i^{X_i}
             +v\partial_{x}^{k}Q=\mu\partial_{x}^{k+1}\Psi+F_4^k,
\end{cases}
\end{equation}
where $$F_3^k=p^{\prime}(v)\partial^{k+1}_{x}(v-\widetilde{v})-\partial^{k+1}_{x}(p(v)-p(\widetilde{v}))
-\partial^k_{x}F_1$$
and
$$F_4^k=v\partial^k_{x}(\Pi-\widetilde{\Pi})-\partial^k_{x}(v\Pi-\widetilde{v}\widetilde{\Pi})
-\partial^k_{x}F_2.$$

Multiplying the above equations by $-p^{\prime}(v)\partial_{x}^k\Phi,\partial_{x}^k\Psi,\frac{1}{\mu}\partial_{x}^kQ$, respectively, and integrating over $\mathbb{R}$, we get
\begin{align}\label{4.11}
  \frac{\dif}{\dif t}\int_\mathbb{R}\left(-\frac{p^{\prime}(v)}{2}(\partial_{x}^k\Phi)^2
  +\frac{1}{2}(\partial_{x}^k\Psi)^2
  +\frac{\tau}{2\mu}(\partial_{x}^kQ)^2\right)\dif x
   +\int_\mathbb{R}\frac{v}{\mu}(\partial_x^kQ)^2\dif x
    =:\sum\limits_{i=1}\limits^{8}R^k_i,
\end{align}
where
\[
R^k_1=-\int_\mathbb{R}\frac{p^{\prime\prime}(v)}{2}v_t(\partial_{x}^k\Phi)^2\dif x, \quad
R^k_2=\int_\mathbb{R}p^{\prime\prime}(v)v_{x}\partial_{x}^k\Phi\partial_{x}^k\Psi \dif x,
\]
\[
R^k_3=-\sum\limits_{i=1}\limits^2
\int_\mathbb{R}\dot{X}_i\partial_{x}^{k+1}\widetilde{v}^{X_i}_ip^{\prime}(v)\partial_{x}^k\Phi \dif x,
\quad
R^k_4=\sum\limits_{i=1}\limits^2
\int_\mathbb{R}\dot{X}_i\partial_{x}^{k+1}\widetilde{u}^{X_i}_i\partial_{x}^k\Psi \dif x,
\]
\[
R^k_5=\frac{\tau}{\mu}\sum\limits_{i=1}\limits^2
\int_\mathbb{R}\dot{X}_i\partial_{x}^{k+1}\widetilde{\Pi}^{X_i}_iQ\dif x,\quad
R^k_6=\int_\mathbb{R}F_3^k\partial_{x}^k\Psi \dif x,\quad
R^k_7=\frac{1}{\mu}\int_\mathbb{R}F_4^k\partial_{x}^kQ\dif x.
\]
Firstly, using Lemmas \ref{le2.1} and Sobolev's imbedding theorem, we have
\begin{align*}
R^k_1=-\int_\mathbb{R}\frac{p^{\prime\prime}(v)}{2}(u-\widetilde{u})_x(\partial_{x}^k\Phi)^2\dif x
-\int_\mathbb{R}\frac{p^{\prime\prime}(v)}{2}\widetilde{u}_x(\partial_{x}^k\Phi)^2\dif x
\le C(\varepsilon_1+\delta_1^2+\delta_2^2)\int_\mathbb{R}(\partial_{x}^k\Phi)^2\dif x.
\end{align*}
Similarly, for $R^k_2$, using Lemma \ref{le2.1}, Sobolev's imbedding theorem and Young's inequality, we have
\begin{align*}
R^k_2\le C(\varepsilon_1+\delta_1^2+\delta_2^2)\left(\int_\mathbb{R}(\partial_{x}^k\Phi)^2\dif x
+\int_\mathbb{R}(\partial_{x}^k\Psi)^2\dif x\right).
\end{align*}
For $R^k_3$, using Lemma \ref{le2.1} and Young's inequality, we have
\begin{align*}
\int_\mathbb{R}\dot{X}_i\partial_{x}^{k+1}\widetilde{v}^{X_i}_xp^{\prime}(v)\partial_{x}^k\Phi \dif x
&\le C\delta_i|\dot{X}_i|^2\int_\mathbb{R}|\partial_{x}^{k+1}\widetilde{v}^{X_i}_i|\dif x
+\frac{C}{\delta_i}\int_\mathbb{R}|\partial_{x}^{k+1}\widetilde{v}^{X_i}_i|(\partial_{x}^k\Phi)^2 \dif x\\
&\le C\delta_i^2|\dot{X}_i|^2
+C\delta_i\int_\mathbb{R}(\partial_{x}^k\Phi)^2 \dif x.
\end{align*}
So, we get
\[
R^k_3\le C\sum\limits_{i=1}\limits^2\left(\delta_i^2|\dot{X}_i|^2
+\delta_i\int_\mathbb{R}(\partial_{x}^k\Phi)^2 \dif x\right).
\]
Similarly, for $R^k_4,R^k_5$, we have
\[
R^k_4\le C\sum\limits_{i=1}\limits^2\left(\delta_i^2|\dot{X}_i|^2
+\delta_i\int_\mathbb{R}(\partial_{x}^k\Psi)^2 \dif x\right),\quad
R^k_5\le C\tau\sum\limits_{i=1}\limits^2\left(\delta_i^2|\dot{X}_i|^2
+\delta_i\int_\mathbb{R}(\partial_{x}^kQ)^2 \dif x\right).
\]
Next, we estimate $R_6^k$. For $k=1$, since
\begin{align*}
  &p^{\prime}(v)(v-\widetilde{v})_{xx}-(p(v)-p(\widetilde{v}))_{xx}\\
   &=-p^{\prime\prime}(v)v^2_{x}-p^{\prime}(v)\widetilde{v}_{xx}
                    +p^{\prime\prime}(\widetilde{v})\widetilde{v}^2_{x}
                    +p^{\prime}(\widetilde{v})\widetilde{v}_{xx}\\
   &=-p^{\prime\prime}(v)(v_{x}-\widetilde{v}_{x})^2
   -2p^{\prime\prime}(v)(v_{x}-\widetilde{v}_{x})\widetilde{v}_{x}
     -(p^{\prime\prime}(v)-p^{\prime\prime}(\widetilde{v}))\widetilde{v}^2_{x}
     -(p^{\prime}(v)-p^{\prime}(\widetilde{v}))\widetilde{v}_{xx},
\end{align*}
using Lemma \ref{le2.1}, we have
\begin{align*}
 &\int_\mathbb{R}(u-\widetilde{u})_{x}(p^{\prime}(v)(v-\widetilde{v})_{xx}-(p(v)-p(\widetilde{v}))_{xx})\dif x\\
      &\leq C\int_\mathbb{R}|(u-\widetilde{u})_{x}||(v-\widetilde{v})_{x}|^2\dif x
      +C\int_\mathbb{R}|(u-\widetilde{u})_{x}||(v-\widetilde{v})_{x}||\widetilde{v}_{x}|\dif x
      +C\int_\mathbb{R}|(u-\widetilde{u})_{x}||v-\widetilde{v}||\widetilde{v}_{x}|\dif x.
\end{align*}
We estimate each term of the right hand side of the above inequality. Using Young's inequality and Sobolev's imbedding theorem, we have
\[
\int_\mathbb{R}|(u-\widetilde{u})_{x}||(v-\widetilde{v})_{x}|^2\dif x\leq
 C\varepsilon_1\int_\mathbb{R}|\partial_{x}\Phi|^2\dif x
\]
and
\begin{align*}
\int_\mathbb{R}|(u-\widetilde{u})_{x}||(v-\widetilde{v})_{x}||\widetilde{v}_{x}|\dif x\leq C(\delta_1^2+\delta_2^2)
(\int_\mathbb{R}|\partial_{x}\Phi|^2\dif x+\int_\mathbb{R}|\partial_{x}\Psi|^2\dif x).
\end{align*}
In a similar way, using Lemma \ref{le2.1} and Young's inequality, we get
\begin{align*}
&\int_\mathbb{R}|(u-\widetilde{u})_{x}||v-\widetilde{v}||(\widetilde{v}_i)^{X_i}_{x}|\dif x\\
&\le C\delta_i\left(\int_\mathbb{R}|\partial_{x}\Psi|^2\dif x
+\int_\mathbb{R}|(\widetilde{v}_i)^{X_i}_{x}||v-\widetilde{v}|^2\dif x\right)\\
&\le  C\delta_i\left(\int_\mathbb{R}|\partial_{x}\Psi|^2\dif x
+\int_\mathbb{R}|(\widetilde{v}_i)^{X_i}_{x}||\phi_i(v-\widetilde{v})|^2\dif x
+\int_\mathbb{R}|(\widetilde{v}_i)^{X_i}_{x}|(1-\phi_i^2)|v-\widetilde{v}|^2\dif x\right)\\
&\le C\delta_i\left(\int_\mathbb{R}|\partial_{x}\Psi|^2\dif x
+G^s(U)
+\varepsilon_1 \delta_i^2e^{-C\delta_i t}\right).
\end{align*}
Hence, we derive that
\begin{align*}
 &\int_\mathbb{R}(u-\widetilde{u})_{x}(p^{\prime}(v)(v-\widetilde{v})_{xx}-(p(v)-p(\widetilde{v}))_{xx})\dif x\\
  &\le C(\varepsilon_1+\delta_1+\delta_2)\int_\mathbb{R}|\partial_{x}\Phi|^2\dif x
  +C(\delta_1+\delta_2)\int_\mathbb{R}|\partial_{x}\Psi|^2\dif x
  +C(\delta_1+\delta_2)G^s+C\sum\limits_{i=1}\limits^2\varepsilon_1 \delta_i^3e^{-C\delta_i t}.
\end{align*}
On the other hand, using Lemma \ref{le2.1}, we derive that for $k=1,2$,
\begin{align*}
\partial_x^{k}F_1&=
\partial_x^{k+1}\left(p(\widetilde{v})-p(\widetilde{v}_1^{X_1})-p(\widetilde{v}_2^{X_2})\right)\\
&\le
C\left(|(\widetilde{v}_1)^{X_1}_x||\widetilde{v}-(\widetilde{v}_1)^{X_1}|
+|(\widetilde{v}_2)^{X_2}_x||\widetilde{v}-(\widetilde{v}_2)^{X_2}|
+|(\widetilde{v}_1)^{X_1}_x||(\widetilde{v}_2)^{X_2}_x|\right).
\end{align*}
So, using Lemmas \ref{le2.1}, \ref{le4.3}, H\"{o}lder inequality and Sobolev's imbedding theorem, we have
\begin{align*}
-&\int_{\mathbb{R}}(u-\widetilde{u})_{x}\partial_{x}F_1\dif x\\
&\le C\varepsilon_1 \left(\left\||(\widetilde{v}_1)^{X_1}_x||\widetilde{v}-(\widetilde{v}_1)^{X_1}|\right\|_{L^2}
+\left\||(\widetilde{v}_2)^{X_2}_x||\widetilde{v}-(\widetilde{v}_2)^{X_2}|\right\|_{L^2}
+\left\||(\widetilde{v}_1)^{X_1}_x||(\widetilde{v}_2)^{X_2}_x|\right\|_{L^2}\right)\\
&\le C\varepsilon_1\delta_1\delta_2(\delta_1^{\frac{1}{2}}+\delta_2^{\frac{1}{2}})
e^{-C\min\{\delta_1,\delta_2\}t}.
\end{align*}
Combining the above estimates, we derive that
\begin{align*}
R_6^1\le &C(\varepsilon_1+\delta_1+\delta_2)\int_\mathbb{R}|\partial_{x}\Phi|^2\dif x
  +C(\delta_1+\delta_2)\int_\mathbb{R}|\partial_{x}\Psi|^2\dif x
  +C(\delta_1+\delta_2)G^s\\
  &+C\varepsilon_1\delta_1\delta_2e^{-C\min\{\delta_1,\delta_2\}t}+C\sum\limits_{i=1}\limits^2\varepsilon_1 \delta_i^3e^{-C\delta_i t}.
\end{align*}
Similarly, for $k=2$, since
\begin{align*}
  &p^{\prime}(v)(v-\widetilde{v})_{xxx}-(p(v)-p(\widetilde{v}))_{xxx}\\
   &=-p^{\prime\prime\prime}(v)(v_{x}-\widetilde{v}_{x})^3
   -(p^{\prime\prime\prime}(v)-p^{\prime\prime\prime}(\widetilde{v}))\widetilde{v}_{x}^3
   -3p^{\prime\prime\prime}(v)\widetilde{v}_{x}^2(v_{\xi}-\widetilde{v}_{x})
   -3p^{\prime\prime\prime}(v)\widetilde{v}_{x}(v_{x}-\widetilde{v}_{x})^2\\
   &\quad-3p^{\prime\prime}(v)(v_{x}-\widetilde{v}_{x})(v_{xx}
   -\widetilde{v}_{xx})-3p^{\prime\prime}(v)\widetilde{v}_{x}(v_{xx}-\widetilde{v}_{xx})
  -3p^{\prime\prime}(v)\widetilde{v}_{xx}(v_{x}-\widetilde{v}_{x})\\
   &\quad-3(p^{\prime\prime}(v)-p^{\prime\prime}(\widetilde{v}))\widetilde{v}_{x}\widetilde{v}_{xx}
   -(p^{\prime}(v)-p^{\prime}(\widetilde{v}))\widetilde{v}_{xxx}.
\end{align*}
Using Lemma \ref{le2.1}, Young's inequality and Sobolev's imbedding theorem, we have
\begin{align*}
 &\int_\mathbb{R}(u-\widetilde{u})_{xx}
     (p^{\prime}(v)(v-\widetilde{v})_{xxx}-(p(v)-p(\widetilde{v}))_{xxx})\dif x\\
      &\leq C\int_\mathbb{R}|(u-\widetilde{u})_{xx}||v-\widetilde{v}||\widetilde{v}_x|\dif x
      +C\int_\mathbb{R}|(u-\widetilde{u})_{xx}||(v-\widetilde{v})_{x}|^3\dif x\\
     &\quad +C\int_\mathbb{R}|(u-\widetilde{u})_{xx}||(v-\widetilde{v})_x||\widetilde{v}_x|\dif x
      +C\int_\mathbb{R}|(u-\widetilde{u})_{xx}||(v-\widetilde{v})_{xx}||\widetilde{v}_x|\dif x\\
& \le C(\varepsilon_1^2+\delta_1+\delta_2)\left(\int_\mathbb{R}|\partial_{x}\Phi|^2\dif x
+\int_\mathbb{R}|\partial_{xx}\Psi|^2\dif x\right)
  +C(\delta_1+\delta_2)\int_\mathbb{R}|\partial_{xx}\Phi|^2\dif x\\
 &\quad +C(\delta_1+\delta_2)G^s+C\sum\limits_{i=1}\limits^2\varepsilon_1 \delta_i^3e^{-C\delta_i t}
\end{align*}
and
\begin{align*}
-&\int_{\mathbb{R}}(u-\widetilde{u})_{xx}\partial_{xx}F_1\dif x\\
&\le C\varepsilon_1 \left(\left\||(\widetilde{v}_1)^{X_1}_x||\widetilde{v}-(\widetilde{v}_1)^{X_1}_x|\right\|_{L^2}
+\left\||(\widetilde{v}_2)^{X_2}_x||\widetilde{v}-(\widetilde{v}_2)^{X_2}_x|\right\|_{L^2}
+\left\||(\widetilde{v}_1)^{X_1}_x||(\widetilde{v}_2)^{X_2}_x|\right\|_{L^2}\right)\\
&\le C\varepsilon_1\delta_1\delta_2(\delta_1^{\frac{1}{2}}+\delta_2^{\frac{1}{2}})
e^{-C\min\{\delta_1,\delta_2\}t}.
\end{align*}
So, we have
\begin{align*}
&R_6^2\le C(\varepsilon_1^2+\delta_1+\delta_2)\left(\int_\mathbb{R}|\partial_{x}\Phi|^2\dif x
+\int_\mathbb{R}|\partial_{xx}\Psi|^2\dif x\right)
  +C(\delta_1+\delta_2)\int_\mathbb{R}|\partial_{xx}\Phi|^2\dif x\\
   &\qquad+C(\delta_1+\delta_2)G^s
  +C\varepsilon_1\delta_1\delta_2e^{-C\min\{\delta_1,\delta_2\}t}+C\sum\limits_{i=1}\limits^2\varepsilon_1 \delta_i^3e^{-C\delta_i t}.
\end{align*}
Similarly, for $R_7^1$, since
\[
v(\Pi-\widetilde{\Pi})_{x}-(v\Pi-\widetilde{v}\widetilde{\Pi})_{x}=-(v-\widetilde{v})\widetilde{\Pi}_{x}
-(v-\widetilde{v})_{x}(\Pi-\widetilde{\Pi})
-(v-\widetilde{v})_{x}\widetilde{\Pi}-\widetilde{v}_{x}(\Pi-\widetilde{\Pi}),
\]
using Lemma \ref{le2.1} and $|\widetilde{\Pi}_i^{X_i}|\le C|(\widetilde{v}_i)^{X_i}_x|$, we have
\begin{align*}
 &\int_\mathbb{R}\frac{(\Pi-\widetilde{\Pi})_{x}}{\mu}
 \left(v(\Pi-\widetilde{\Pi})_{x}-(v\Pi-\widetilde{v}\widetilde{\Pi})_{x}\right)\dif x\\
 &\le C\int_\mathbb{R}\frac{|(\Pi-\widetilde{\Pi})_{x}|}{\mu}|v-\widetilde{v}||\widetilde{v}_{x}|\dif x
 + C\int_\mathbb{R}\frac{|(\Pi-\widetilde{\Pi})_{x}|}{\mu}|(v-\widetilde{v})_{x}||\Pi-\widetilde{\Pi}|\dif x\\
 &\quad+ C\int_\mathbb{R}\frac{|(\Pi-\widetilde{\Pi})_{x}|}{\mu}|(v-\widetilde{v})_{x}||\widetilde{\Pi}|\dif x
 + C\int_\mathbb{R}\frac{|(\Pi-\widetilde{\Pi})_{x}|}{\mu}|\Pi-\widetilde{\Pi}||\widetilde{v}_x|\dif x\\
 &\le C(\varepsilon_1+\delta_1+\delta_2)\left(\int_\mathbb{R}\frac{v}{\mu}(\partial_xQ)^2\dif x+G\right)
 +C(\delta_1+\delta_2)\int_\mathbb{R}(\partial_x\Phi)^2\dif x
 +C(\delta_1+\delta_2)G^s
 +C\sum\limits_{i=1}\limits^2\varepsilon_1 \delta_i^3e^{-C\delta_i t}.
\end{align*}
On the other hand, for $k=1,2$, using Lemma \ref{le2.1}, we get
\begin{align*}
\partial_x^{k}F_2&=
\partial_x^{k}\left(\left(\widetilde{v}^{X_2}_2-v_m\right)\widetilde{\Pi}_1^{X_1}
+\left(\widetilde{v}^{X_1}_1-v_m\right)\widetilde{\Pi}_2^{X_2}\right)\\
&\le
C\left(|(\widetilde{v}_1)^{X_1}_x||\widetilde{v}-(\widetilde{v}_1)^{X_1}_x|
+|(\widetilde{v}_2)^{X_2}_x||\widetilde{v}-(\widetilde{v}_2)^{X_2}_x|
+|(\widetilde{v}_1)^{X_1}_x||(\widetilde{v}_2)^{X_2}_x|\right).
\end{align*}
Then, using Lemma \ref{le4.3} and Young's inequality, we have
\begin{align*}
 &\int_\mathbb{R}\frac{(\Pi-\widetilde{\Pi})_{x}}{\mu}
 \left(\left(\widetilde{v}^{X_2}_2-v_m\right)\widetilde{\Pi}_1^{X_1}
+\left(\widetilde{v}^{X_1}_1-v_m\right)\widetilde{\Pi}_2^{X_2}\right)_x\dif x\\
&\le C\|\partial_x F_2\|_{L^2} +\frac{1}{16}\int_\mathbb{R}\frac{v}{\mu}(\partial_xQ)^2\dif x\\
&\le C\delta_1^2\delta_2^2e^{-C\min\{\delta_1,\delta_2\}t}
+\frac{1}{16}\int_\mathbb{R}\frac{v}{\mu}(\partial_xQ)^2\dif x.
\end{align*}
Combining the above estimates, we derive that
\begin{align*}
 &R_7^1 \le \frac{1}{8}\int_\mathbb{R}\frac{v}{\mu}(\partial_xQ)^2\dif x
 +C(\varepsilon_1+\delta_1+\delta_2)G
 +C(\delta_1+\delta_2)\left(\int_\mathbb{R}(\partial_x\Phi)^2\dif x+G^s\right)\\
 &+C\delta_1^2\delta_2^2e^{-C\min\{\delta_1,\delta_2\}t}
 +C\sum\limits_{i=1}\limits^2\varepsilon_1 \delta_i^3e^{-C\delta_i t}.
\end{align*}
Next, for $R_7^2$, first we have
\begin{align*}
v(\Pi-\widetilde{\Pi})_{xx}-(v\Pi-\widetilde{v}\widetilde{\Pi})_{xx}=
&-(v-\widetilde{v})_{xx}(\Pi-\widetilde{\Pi})
-(v-\widetilde{v})_{xx}\widetilde{\Pi}-\widetilde{v}_{xx}(\Pi-\widetilde{\Pi})
-(v-\widetilde{v})\widetilde{\Pi}_{xx}\\
&-2(v-\widetilde{v})_{x}(\Pi-\widetilde{\Pi})_{x}-2(v-\widetilde{v})_{x}\widetilde{\Pi}_{x}
-2\widetilde{v}_{x}(\Pi-\widetilde{\Pi})_{x}.
\end{align*}
Specially, using H\"{o}lder inequality, Sobolev's imbedding theorem and Young's inequality, we have
\begin{align*}
&\int_\mathbb{R}(v-\widetilde{v})_{xx}(\Pi-\widetilde{\Pi})(\Pi-\widetilde{\Pi})_{xx}\dif x\\
&\leq C
\|(v-\widetilde{v})_{\xi\xi}\|_{L^2}
\left(\int_\mathbb{R}\left(\frac{v}{\mu}\right)^2(\Pi-\widetilde{\Pi})^2(\Pi-\widetilde{\Pi})_{xx}^2\dif x\right)^{\frac{1}{2}}\\
&\leq C\varepsilon_1|\sqrt{\frac{v}{\mu}}(\Pi-\widetilde{\Pi})|_{L^{\infty}}
\|\sqrt{\frac{v}{\mu}}(\Pi-\widetilde{\Pi})_{xx}\|_{L^2}\\
&\leq C\varepsilon_1\left(\|\sqrt{\frac{v}{\mu}}(\Pi-\widetilde{\Pi})\|_{H^1}^2+
\|\sqrt{\frac{v}{\mu}}(\Pi-\widetilde{\Pi})_{xx}\|_{L^2}^2\right).
\end{align*}
So, we have
\begin{align*}
 &\int_\mathbb{R}\frac{(\Pi-\widetilde{\Pi})_{xx}}{\mu}
 \left(v(\Pi-\widetilde{\Pi})_{xx}-(v\Pi-\widetilde{v}\widetilde{\Pi})_{xx}\right)\dif x\\
 &\le C(\varepsilon_1+\delta_1+\delta_2)\left\|\sqrt{\frac{v}{\mu}}Q\right\|_{H^2}
 +(\delta_1+\delta_2)\left\|\partial_x\Phi\right\|_{H^1}+(\delta_1+\delta_2)G^s+C\sum\limits_{i=1}\limits^2\varepsilon_1 \delta_i^3e^{-C\delta_i t}.
\end{align*}
On the other hand, using the definition of $F_2$ and Young's inequality, we have
\begin{align*}
 &\int_\mathbb{R}\frac{(\Pi-\widetilde{\Pi})_{xx}}{\mu}
 \left(\left(\widetilde{v}^{X_2}_2-v_m\right)\widetilde{S}_1^{X_1}
+\left(\widetilde{v}^{X_1}_1-v_m\right)\widetilde{\Pi}_2^{X_2}\right)_{xx}\dif x\\
&\le C\|\partial_{xx}F_2\|_{L^2}
+\frac{1}{16}\int_\mathbb{R}\frac{v}{\mu}(\partial_{xx}Q)^2\dif x\\
&\le C\delta_1^2\delta_2^2e^{-C\min\{\delta_1,\delta_2\}t}
+\frac{1}{16}\int_\mathbb{R}\frac{v}{\mu}(\partial_{xx}Q)^2\dif x.
\end{align*}
Therefore, we get
\begin{align*}
R_7^2&\le \frac{1}{8}\int_\mathbb{R}\frac{v}{\mu}(\partial_{xx}Q)^2\dif x
 +C(\varepsilon_1+\delta_1+\delta_2)\left\|\sqrt{\frac{v}{\mu}}Q\right\|_{H^1}
  +C(\delta_1+\delta_2)\left(G^s+\left\|\partial_x\Phi\right\|_{H^1}\right)\\
  &+C\sum\limits_{i=1}\limits^2\varepsilon_1 \delta_i^3e^{-C\delta_i t}
 +C\delta_1^2\delta_2^2e^{-C\min\{\delta_1,\delta_2\}t}.
\end{align*}
Finally, integrating the equality \eqref{4.11} over $[0,t]$, using the above estimates, we complete the proof of this lemma.
\end{proof}

\subsection{Dissipative estimates}
In the following lemmas, we give the dissipative estimates of given solutions to system \eqref{4.1}.
\begin{lemma}\label{Le4.8}
Under the hypotheses of Proposition \ref{p1}, there exist $C,\nu>0$ (independent of $\tau$ and $T$) such that for all $t\in(0,T]$, we have
\begin{equation}\label{26.3}
\begin{aligned}
&(1-\kappa-C(\delta_0+\varepsilon_1))\int_0^t\|\sqrt{|p^\prime(v)|}\partial_{x}\Phi\|_{H^1}^2\dif t\\
&\leq\nu\|\Psi\|_{H^1}^2
+C(\nu)\|\partial_{x}\Phi\|_{H^1}^2
     +C\int_0^t\sum\limits_{i=1}\limits^2\delta_i^2|\dot{X}_i(t)|^2\dif t+C\int_0^t\|\partial_{x}Q\|_{H^1}^2\dif t\\
     &+C\left(\|\Psi_0\|_{H^1}^2+\|\partial_{x}\Phi_0\|_{H^1}^2\right)
     +(1+C\delta_0)\int_0^t\|\partial_{x}\Psi\|_{H^1}^2\dif t
     +C\delta_0\int_0^tG^s(U)\dif t+C\delta_0,
\end{aligned}
\end{equation}
where $\kappa$ is small constant to be determined later and $G^s(U)$ is defined in \eqref{3.7}.
\end{lemma}
\begin{proof}
Multiplying the equation $\eqref{4.10}_2$ by $\partial_{x}^{k+1}\Phi$ for $k=0 , 1,$ and integrating over $(0,t)\times \mathbb{R}$, we get
\[
\int_0^t\int_{\mathbb{R}}-p^{\prime}(v)\left(\partial_{x}^{k+1}\Phi\right)^2\dif x \dif t=:\sum\limits_{i=0}\limits^{4}M_i^k,
\]
where
\[
M_1^k=\int_0^t\int_\mathbb{R}\partial_{t}\partial_{x}^k\Psi\partial_{x}^{k+1}\Phi \dif x \dif t,
\quad M_2^k=-\sum\limits_{i=1}\limits^2\int_0^t\int_\mathbb{R}\dot{X_i}\partial_{x}^{k+1}\widetilde{u}_i^{X_i}
\partial_{x}^{k+1}\Phi \dif x \dif t,
\]
\[
M_3^k=-\int_0^t\int_\mathbb{R}\partial_{x}^{k+1}Q\partial_{x}^{k+1}\Phi \dif x \dif t,
\quad M_4^k=-\int_0^t\int_\mathbb{R}F_3^k\partial_{x}^{k+1}\Phi \dif x \dif t.
\]
Firstly, doing integration by part and using equation $\eqref{4.10}_1$, we get
\begin{align*}
&M_1^k=\int_\mathbb{R}\left(\partial_{x}^k\Psi(t)\partial_{x}^{k+1}\Phi(t) -\partial_{x}^k\Psi_0\partial_{x}^{k+1}\Phi_0 \right)\dif x
-\int_0^t\int_\mathbb{R}\partial_{x}^k\Psi
      \left(\sum\limits_{i=1}\limits^2\dot{X}_i\partial_{x}^{k+2}\widetilde{v}_i^{X_i}
      +\partial_{x}^{k+2}\Psi
      \right)\dif x \dif t.
\end{align*}
Since
\[
	\int_\mathbb{R}\partial_{x}^k\Psi(t)\partial_{x}^{k+1}\Phi(t) \dif x\leq \nu\int_\mathbb{R}(\partial_{x}^k\Psi(t))^2 \dif x +C(\nu)\int_\mathbb{R}(\partial_{x}^{k+1}\Phi(t))^2 \dif x,
	\]
		\[
	\int_\mathbb{R}\partial_{x}^k\Psi_0\partial_{x}^{k+1}\Phi_0 \dif x\leq C\int_\mathbb{R}(\partial_{x}^k\Psi_0)^2 \dif x +C\int_\mathbb{R}(\partial_{x}^{k+1}\Phi_0)^2 \dif x,
	\]
\begin{align*}
-\int_0^t\int_\mathbb{R}\partial_{x}^k\Psi
      \left(\sum\limits_{i=1}\limits^2\dot{X}_i\partial_{x}^{k+2}\widetilde{v}_i^{X_i}\right)\dif x \dif t
&=\int_0^t\int_\mathbb{R}\partial_{x}^{k+1}\Psi
      \left(\sum\limits_{i=1}\limits^2\dot{X}_i\partial_{x}^{k+1}\widetilde{v}_i^{X_i}\right)\dif x \dif t\\
&\le C\sum\limits_{i=1}\limits^2\left(\delta_i^2\int_0^t|\dot{X}_i|^2\dif t
+\delta_i\int_0^t\int_\mathbb{R}(\partial_{x}^{k+1}\Psi)^2 \dif x\dif t\right)
\end{align*}
and
\[
-\int_0^t\int_\mathbb{R}\partial_{x}^k\Psi\partial_{x}^{k+2}\Psi\dif x \dif t
=\int_0^t\int_\mathbb{R}\left(\partial_{x}^{k+1}\Psi\right)^2\dif x \dif t,
\]
we obtain
\begin{align*}
M_1^k\leq&\nu\int_\mathbb{R}\left(\partial_{x}^k\Psi(t)\right)^2\dif x
+C(\nu)\int_\mathbb{R}\left(\partial_{x}^{k+1}\Phi(t)\right)^2\dif x
+(1+C\delta_0)\int_0^t\int_\mathbb{R}\left(\partial_{x}^{k+1}\Psi\right)^2\dif x \dif t\\
&+C\left(\sum\limits_{i=1}\limits^2\delta_i^2\int_0^t|\dot{X}_i|^2\dif t
+\int_\mathbb{R}\left(\partial_{x}^k\Psi_0\right)^2\dif x
+\int_\mathbb{R}\left(\partial_{x}^{k+1}\Phi_0\right)^2\dif x\right).
\end{align*}
Secondly, for $M_2^k$ and $M_3^k$, using Young's inequality and Lemma \ref{le2.1}, we have
\begin{align*}
M_2^k\le  C\sum\limits_{i=1}\limits^2\left(\delta_i^2\int_0^t|\dot{X}_i|^2\dif t
+\delta_i\int_0^t\int_\mathbb{R}|p^\prime(v)|(\partial_{x}^{k+1}\Phi)^2 \dif x\dif t\right)
\end{align*}
and
\[
M_3^k\le
 \kappa\int_0^t\int_{\mathbb{R}}-p^{\prime}(v)\left(\partial_{x}^{k+1}\Phi\right)^2\dif x \dif t
   +C\int_0^t\int_\mathbb{R}\left(\partial_{x}^{k+1}Q\right)^2\dif x \dif t.
\]
Next, we estimate $M_4^k$. For $k=0$,  we have
\[
F_3^0=-(p^{\prime}(v)-p^{\prime}(\widetilde{v}))\widetilde{v}_{x}
-\left(p(\widetilde{v})_x-p(\widetilde{v}_1)^{X_1}_x-p(\widetilde{v}_2)^{X_2}_x\right),
\]
which gives
\begin{align*}
M_4^0\le C(\delta_1+\delta_2)\left(\int_0^t\int_\mathbb{R}|\partial_{x}\Phi|^2\dif x\dif t
  +\int_0^tG^s\dif t\right)
  +C\varepsilon_1\max\{\delta_1,\delta_2\}.
\end{align*}
For $k=1$, the term $F_3^1$ is the same as in Lemma \ref{Le4.7}. Therefore, using similar estimates as in Lemma \ref{Le4.7}, we have
\begin{align*}
M_4^1&\le C(\varepsilon_1+\delta_1+\delta_2)
\int_0^t\|\partial_{x}\Phi\|_{H^1}^2\dif t
 +C(\delta_1+\delta_2)\int_0^tG^s\dif t
  +C\varepsilon_1\max\{\delta_1,\delta_2\}.
\end{align*}
Therefore, combining the above estimates, we get the desired results.
\end{proof}

\begin{lemma}\label{Le4.9}
Under the hypotheses of Proposition \ref{p1}, there exist $C,\nu_1>0$ (independent of $\tau$ and $T$) such that for all $t\in(0,T]$, we have
\begin{equation}\label{26.4}
\begin{aligned}
&\frac{\mu}{2}(1-C\delta_0)\int_0^t\|\partial_{x}\Psi\|_{H^1}^2\dif t
\leq\nu_1\tau\|Q\|_{H^1}^2
+C(\nu_1)\|\partial_{x}\Psi\|_{H^1}^2
     +C\int_0^t\sum\limits_{i=1}\limits^2\delta_i^2|\dot{X}_i(t)|^2\dif t\\
     &\quad+\mu(\kappa+C\delta_0)\int_0^t\|\sqrt{|p^\prime(v)|}\partial_{x}\Phi\|_{H^1}^2\dif t
     +C\left(\tau\|Q_0\|_{H^1}^2+\|\partial_{x}\Psi_0\|_{H^1}^2\right)\\
&\quad+C\int_0^t\|\partial_{x}Q\|_{H^1}^2\dif t+C(\varepsilon_1+\delta_0)\int_0^t G^s(U)\dif t
+\frac{v_m+C(\delta_0+\varepsilon_1)}{2}\int_0^t G(U)\dif t+C\delta_0,
\end{aligned}
\end{equation}
where $\kappa$ is small constant to be determined later, $G^s(U), G(U)$ are defined in \eqref{3.7} and \eqref{26.1b}, respectively.
\end{lemma}
\begin{proof}
Multiplying the equation $\eqref{4.10}_3$ by $\partial_{x}^{k+1}\Psi$ for $k=0 , 1$, and integrating over $(0,t)\times\mathbb{R}$, we get
\[
\int_0^t\int_{\mathbb{R}}\mu\left(\partial_{x}^{k+1}\Psi\right)^2\dif x \dif t
=:\sum\limits_{i=0}\limits^{4}N_i^k,
\]
where
\[
N_1^k=\int_0^t\int_\mathbb{R}\tau\partial_{t}\partial_{x}^kQ\partial_{x}^{k+1}\Psi \dif x \dif t,
\quad
N_2^k=-\tau\sum\limits_{i=1}\limits^2\dot{X}_i\partial_x^{k+1}\widetilde{\Pi}_i^{X_i}\partial_{x}^{k+1}\Psi\dif x\dif t,
\]
\[
N_3^k=\int_0^t\int_\mathbb{R}v\partial_{x}^{k}Q\partial_{x}^{k+1}\Psi \dif x \dif t,
\quad N_4^k=-\int_0^t\int_\mathbb{R}F_4^k\partial_{x}^{k+1}\Psi \dif x \dif t.
\]
Firstly, doing integration by part and using equation $\eqref{4.10}_2$, we get
\begin{align*}
N_1^k=&\tau\int_\mathbb{R}\partial_{x}^kQ(t)\partial_{x}^{k+1}\Psi(t) \dif x-\tau\int_\mathbb{R}\partial_{x}^kQ_0\partial_{x}^{k+1}\Psi_0 \dif x\\
      &-\tau\int_0^t\int_\mathbb{R}\partial_{x}^kQ
      \left(\sum\limits_{i=1}\limits^2\dot{X}_i\partial_{x}^{k+2}\widetilde{u}_i^{X_i}
      -\partial_{x}^{k+2}(p(v)-p(\widetilde{v}))+\partial_{x}^{k+2}Q+\partial_{x}^{k+2}F_1\right)\dif x \dif t.
\end{align*}
Integrating by part and using lemma \ref{le2.1} and Young's inequality, we get
\[
	\tau\int_\mathbb{R}\partial_{x}^kQ(t)\partial_{x}^{k+1}\Psi(t) \dif x\leq 	\nu_1\tau\int_\mathbb{R}(\partial_{x}^kQ(t))^2 \dif x
	+C(\nu_1)\tau\int_\mathbb{R}(\partial_{x}^{k+1}\Psi(t))^2 \dif x,
	\]
	\[
	\tau\int_\mathbb{R}\int_\mathbb{R}\partial_{x}^kQ_0\partial_{x}^{k+1}\Psi_0 \dif x\leq C\tau\int_\mathbb{R}(\partial_{x}^kQ_0)^2 \dif x
	+C\tau\int_\mathbb{R}(\partial_{x}^{k+1}\Psi_0)^2 \dif x,
	\]
\begin{align*}
-\tau\int_0^t\int_\mathbb{R}\partial_{x}^kQ
     \sum\limits_{i=1}\limits^2\dot{X}_i\partial_{x}^{k+2}\widetilde{u}_i^{X_i}\dif x \dif t
&=\tau\int_0^t\int_\mathbb{R}\partial_{x}^{k+1}Q
     \sum\limits_{i=1}\limits^2\dot{X}_i\partial_{x}^{k+1}\widetilde{u}_i^{X_i}\dif x \dif t\\
&\le C\tau\sum\limits_{i=1}\limits^2\left(\delta_i^2\int_0^t|\dot{X}_i|^2\dif t
+\delta_i\int_0^t\int_\mathbb{R}(\partial_{x}^{k+1}Q)^2 \dif x\dif t\right),
\end{align*}
\begin{align*}
-\tau\int_0^t\int_\mathbb{R}\partial_{x}^kQ\partial_{x}^{k+2}Q\dif x \dif t
=\tau\int_0^t\int_\mathbb{R}\left(\partial_{x}^{k+1}Q\right)^2\dif x \dif t,
\end{align*}
and
\begin{align*}
&\tau\int_0^t\int_\mathbb{R}\partial_{x}^kQ\partial_{x}^{k+2}(p(v)-p(\widetilde{v}))\dif x \dif t\\
&=-\tau\int_0^t\int_\mathbb{R}\partial_{x}^{k+1}Q\partial_{x}^{k+1}(p(v)-p(\widetilde{v}))\dif x \dif t\\
&\le \mu\kappa\int_0^t\int_\mathbb{R}\left(\partial_{x}^{k+1}\Phi\right)^2\dif x \dif t
+C\tau^2\int_0^t\int_\mathbb{R}\left(\partial_{x}^{k+1}Q\right)^2\dif x \dif t
+C(\delta_1+\delta_2)\int_0^tG^s\dif t
+C\max\{\delta_1,\delta_2\}.
\end{align*}
Recalling the estimates of $\partial_x^{k+1} F_1, k=0, 1$ in Lemma \ref{Le4.7} and \eqref{4.9}, we get
\begin{align*}
-\tau\int_0^t\int_\mathbb{R}\partial_{x}^kQ
      \partial_{x}^{k+2}F_1 \dif x \dif t&=\tau\int_0^t\int_\mathbb{R}\partial_{x}^{k+1}Q
      \partial_{x}^{k+1}F_1 \dif x \dif t\\
     & \leq C\int_0^t\int_\mathbb{R}\left(\partial_{x}^{k+1}Q\right)^2
       \dif x \dif t+C\tau^2\int_0^t\|\partial_{x}^{k+1}F_1\|_{L^2}^2\dif t\\
      &\leq C\int_0^t\int_\mathbb{R}\left(\partial_{x}^{k+1}Q\right)^2
       \dif x \dif t+C\max\{\delta_1,\delta_2\}.
\end{align*}
Therefore, we conclude that
\begin{align*}
N_1^k&\le\nu_1\tau\int_\mathbb{R}\left(\partial_{x}^kQ(t)\right)^2\dif x
+C(\nu_1)\int_\mathbb{R}\left(\partial_{x}^{k+1}\Psi(t)\right)^2\dif x
+C\int_0^t\int_\mathbb{R}\left(\partial_{x}^{k+1}Q\right)^2\dif x \dif t
+C\tau\sum\limits_{i=1}\limits^2\delta_i^2\int_0^t|\dot{X}_i|^2\dif t\\
&+C\tau\int_\mathbb{R}\left(\partial_{x}^kQ_0\right)^2\dif x
+C\int_\mathbb{R}\left(\partial_{x}^{k+1}\Psi_0\right)^2\dif x
+C\delta_0\int_0^tG^s\dif t+\mu\kappa\int_0^t\int_\mathbb{R}\left(\partial_{x}^{k+1}\Phi\right)^2\dif x \dif t
+C\max\{\delta_1,\delta_2\}.
\end{align*}
Secondly, for $N_2^k$, we have
\begin{align*}
N_2^k\leq C\tau^2\sum\limits_{i=1}\limits^2\delta_i^2\int_0^t|\dot{X}_i|^2\dif t
+C(\delta_1+\delta_2)\int_0^t\int_\mathbb{R}\left(\partial_{x}^{k+1}\Psi\right)^2\dif x \dif t.
\end{align*}
For $N_3^k$, using Young's inequality and noting that $|v_m-v_-|\le C\delta_1, |v_m-v_+|\le C\delta_2$, we have
\[
N_3^k\leq \frac{\mu}{2}\int_0^t\int_{\mathbb{R}}\left(\partial_{x}^{k+1}\Psi\right)^2\dif x \dif t
   +\frac{v_m+C(\delta_0+\varepsilon_1)}{2}\int_0^t\int_\mathbb{R}\frac{v}{\mu}\left(\partial_{x}^{k}Q\right)^2\dif x \dif t.
\]
$N_4^k$ can be estimated in the same way as in Lemma \ref{Le4.7}. Specially, for $k=0$, we have
\[
F_4^0=-(v-\widetilde{v})\widetilde{\Pi}-\left(\widetilde{v}^{X_2}_2-v_m\right)\widetilde{\Pi}_1^{X_1}
-\left(\widetilde{v}^{X_1}_1-v_m\right)\widetilde{\Pi}_2^{X_2}.
\]
So, we have
\[
N_4^0\le C(\delta_1+\delta_2)\int_0^t\int_\mathbb{R}\left(\partial_{x}\Psi\right)^2\dif x \dif t+C(\delta_1+\delta_2)\int_0^tG^s\dif t
+C\max\{\delta_1,\delta_2\}.
\]
For $k=1$, the term $F_3^1$ is the same as in Lemma \ref{Le4.7}. Thus, using similar estimates as in Lemma \ref{Le4.7}, we have
\begin{align*}
N_4^1\le& C(\varepsilon_1+\delta_1+\delta_2)
\left(\int_0^t\int_\mathbb{R}\left(\partial_{xx}\Psi\right)^2\dif x \dif t+\int_0^tG(U)\dif t\right)\\
&+C(\delta_1+\delta_2)\int_0^t\int_\mathbb{R}\left(\partial_{x}\Phi\right)^2\dif x \dif t
+C\max\{\delta_1,\delta_2\}.
\end{align*}
Therefore, combining the above estimates, we get the desired results.
\end{proof}

Next, by using Lemma \ref{le4.2}, Lemma \ref{Le4.7}, Lemma \ref{Le4.8} and Lemma \ref{Le4.9}, we are able to prove Proposition \ref{p1}.

{\bf{Proof of Proposition \ref{p1}}}: For arbitrary constants $C_2, C_3, C_4>0$, we multiply equations \eqref{26.1}, \eqref{26.3} and \eqref{26.4} by $C_2$, $\frac{3\mu}{4v_m}C_3$ and $\frac{2}{v_m}C_4$, respectively. Combining the resulting equations with Lemma \ref{Le4.7} yields
\begin{equation}\label{guan1-1}
\begin{aligned}
&\frac{1}{2}\|\Psi\|^2_{L^2}
+\int_{\mathbb R}H(v|\widetilde v)\dif x
 +\frac{\tau}{2\mu}\|Q\|^2_{L^2}+\|\partial_{x}\Phi\|^2_{H^1}+\|\partial_{x}\Psi\|^2_{H^1}+\tau\|\partial_{x}Q\|^2_{H^1}\\
&+\sum\limits_{i=1}\limits^2(\frac{\delta_i}{4M}-C\delta_i^2)\int_0^t|\dot{X}_i|^2\dif t+C_1(1-C(\delta_0+\varepsilon_1))\int_0^t G^s(U)
\dif t
+C_5\int_0^t
G(U)\dif t\\
&+C_6\int_0^t\|\sqrt{|p^\prime(v)|}\partial_{x}\Phi\|_{H^1}^2\dif t
+C_7\int_0^t\|\partial_{x}\Psi\|_{H^1}^2\dif t
+\int_0^t\|\partial_{x}Q\|_{H^1}^2\dif t\\
 &\leq C\left(\|\Phi_0(\cdot)\|^2_{H^2}+\|\Psi_0(\cdot)\|^2_{H^2}
 +\tau\|Q_0(\cdot)\|^2_{H^2}\right)
 +\nu\frac{3\mu}{4v_m}C_3\|\Psi\|_{L^2}^2+\nu_1\frac{2C_4}{v_m}\tau\|Q\|_{L^2}^2,
\end{aligned}
\end{equation}
	where \begin{align*}
	&	C_5=C_2(1-\kappa-C(\delta_1^{\frac{1}{2}}+\delta_2^{\frac{1}{2}}))-C_4(1+C(\delta_0+\varepsilon_1)),\\
	&C_6=\frac{3\mu}{4v_m}\left(C_3(1-\kappa-C(\delta_0+\varepsilon_1))-C_2\frac{1+\kappa+C\varepsilon_1}{1-C\varepsilon_1}\right),\\
	&C_7=\frac{\mu}{v_m}\left(C_4(1-C\delta_0)-\frac{3}{4}C_3(1+C\delta_0)\right).
\end{align*}  
We now choosing $\delta_1, \delta_2, \delta_0, \varepsilon_1, \kappa, \nu_1, \nu_2$ sufficiently small such that 
\[\frac{1}{4M}>C\delta_i,\quad 1>C(\varepsilon_1+\delta_0),\quad  \nu_1\frac{3\mu C_3}{4v_m}<\frac{1}{2},\quad \nu_2\frac{2C_4}{v_m}<\frac{1}{2\mu} \quad \text{where}\quad i=1,2.\] 
Furthermore, to assure $C_5, C_6, C_7>0$, it is sufficient to require $C_3>C_2>C_4>\frac{3}{4}C_3$. Therefore, combining the above results and using Lemma \ref{le4.1}, the proof of the Proposition \ref{p1} is completed.

\section{Proof of Theorem \ref{th1.2}}
In this section, we show the Theorem \ref{th1.2} by use of the uniform estimates obtained in Section 4 and usual compactness arguments.
Firstly,  according to Theorem \ref{th1}, we get
\begin{align*}
\sup_{0\le t<+\infty}\|(\Phi^{\tau}, \Psi^{\tau}, \sqrt{\tau}Q^{\tau}) (t,\cdot)\|_{H^2}^2+\int_0^{+\infty}\left(\|(\Phi^{\tau}_{x}, \Psi^{\tau}_{x})\|_{H^1}^2+\|Q^{\tau}\|_{H^2}^2\right)\dif t
\le C_0E(0)+C_0\delta_0,
\end{align*}
where $
E(0)=\|(\Phi^\tau,\Psi^\tau,\sqrt{\tau}Q^\tau)(0, \cdot)\|_{H^2}$, $\Phi^{\tau}=v^{\tau}-\widetilde v^{\tau}, \Psi^{\tau}=u^{\tau}-\widetilde u^{\tau}, Q^{\tau}=\Pi^{\tau}-\widetilde \Pi^{\tau}$, $C_0$ is a constant independent of $\tau$ and $\widetilde v^{\tau}=(\widetilde v_1)^{\tau}+(\widetilde v_2)^{\tau}-v_m, \widetilde u^{\tau}=(\widetilde u_1)^{\tau}+(\widetilde u_1)^{\tau}-u_m, \widetilde \Pi^{\tau}=(\widetilde \Pi_1)^{\tau}+(\widetilde \Pi_2)^{\tau}$ are the compose waves of system \eqref{1.3}, respectively. Thus, there exist $(\Phi^0, \Psi^0)\in L^{\infty}((0,\infty);H^2)$ and $Q^0\in L^2((0, \infty);H^2)$ such that
\begin{equation}\label{5.1}
\begin{aligned}
(\Phi^{\tau}, \Psi^{\tau})\rightharpoonup(\Phi^0, \Psi^0)\qquad weak-* \quad in \quad L^{\infty}((0,\infty);H^2),\\
Q^{\tau}\rightharpoonup Q^0 \qquad weakly- \quad in \quad  L^2((0, \infty);H^2).
\end{aligned}
\end{equation}
Secondly, from Lemma \ref{le2.1}, we get
\[
\begin{aligned}
&\|(\widetilde v_1)^{\tau}-v_m\|_{L^2(\mathbb{R}^+)}+\|(\widetilde v_1)^{\tau}-v_-\|_{L^2(\mathbb{R}^-)}
+\|(\widetilde v_1)^{\tau}_{x}\|_{H^2}\le C,\\
&\|(\widetilde v_2)^{\tau}-v_+\|_{L^2(\mathbb{R}^+)}+\|(\widetilde v_2)^{\tau}-v_m\|_{L^2(\mathbb{R}^-)}
+\|(\widetilde v_2)^{\tau}_{x}\|_{H^2}\le C,\\
&\|(\widetilde u_1)^{\tau}-u_m\|_{L^2(\mathbb{R}^+)}+\|(\widetilde u_1)^{\tau}-u_-\|_{L^2(\mathbb{R}^-)}
+\|(\widetilde u_1)^{\tau}_{x}\|_{H^2}\le C,\\
&\|(\widetilde u_2)^{\tau}-u_+\|_{L^2(\mathbb{R}^+)}+\|(\widetilde u_2)^{\tau}-u_m\|_{L^2(\mathbb{R}^-)}
+\|(\widetilde u_2)^{\tau}_{x}\|_{H^2}\le C,\\
&\|(\widetilde \Pi_1)^{\tau}\|_{H^2}\le C,\qquad \|(\widetilde \Pi_2)^{\tau}\|_{H^2}\le C,
\end{aligned}
\]
where $C$ independent of $\tau$.\\
Then, using compactness theorem, for $i=1,2$ and any $T>0$, we have
\[
\begin{aligned}
(\widetilde v_i)^{\tau}\rightarrow(\widetilde v_i)^0,\quad(\widetilde u_i)^{\tau}\rightarrow(\widetilde u_i)^0,\qquad strongly\quad in \quad C([0,T]; H_{loc}^2),\\
(\widetilde \Pi_i)^{\tau}\rightharpoonup(\widetilde \Pi_i)^0,\qquad weakly-\quad in \quad L^\infty((0,\infty), H^2).
\end{aligned}
\]
In addition, let $\tau\rightarrow0$ in \eqref{2.2}, we have $\tau \widetilde \Pi_i\rightarrow 0$ in $D^\prime((0,\infty)\times \mathbb R)$ and
\[
(\widetilde \Pi_i)^{\tau}\rightharpoonup\mu\frac{(\widetilde u_i)_{x}^0}{(\widetilde v_i)^0}:=(\widetilde \Pi_i)^{0},\qquad  in \quad \mathcal D^\prime((0,\infty)\times \mathbb R)
\]
and we know that $(\widetilde v_i)^0,(\widetilde u_i)^0$ are the $i$-traveling wave solutions of classical Navier-Stokes equations.
Therefore,  for any $T>0$, we have
\begin{equation}\label{5.2}
\begin{aligned}
\widetilde v^{\tau}=(\widetilde v_1)^{\tau}+(\widetilde v_2)^{\tau}-v_m\rightarrow(\widetilde v_1)^{0}+(\widetilde v_2)^{0}-v_m=:\widetilde v^{0}, \quad strongly\quad in \quad C([0, T]; H_{loc}^2),\\
 \widetilde u^{\tau}=(\widetilde u_1)^{\tau}+(\widetilde u_2)^{\tau}-u_m\rightarrow(\widetilde u_1)^{0}+(\widetilde u_2)^{0}-u_m=:\widetilde u^{0}, \quad strongly\quad in \quad C([0, T]; H_{loc}^2),\\
  \widetilde \Pi^{\tau}=(\widetilde \Pi_1)^{\tau}+(\widetilde \Pi_2)^{\tau}\rightharpoonup(\widetilde \Pi_1)^{0}+(\widetilde \Pi_2)^{0}=:\widetilde \Pi^{0},\qquad weakly- \quad in \quad L^\infty((0, \infty); H_{loc}^2).
\end{aligned}
\end{equation}
Finally, for any $T>0$, using \eqref{4.1}, we know that $\Phi_t^{\tau}$ and $\Psi_t^{\tau}$ are bounded in $L^2((0,T);H^1)$. Furthermore, using compactness theorem, for any $\alpha>0$, $(\Phi^{\tau},\Psi^{\tau})$ are relatively compact in $C([0,T];H_{loc}^{2-\alpha})$. Then, as $\tau\rightarrow0$, we have
\[
(\Phi^{\tau},\Psi^{\tau})\rightarrow(\Phi^0, \Psi^0)\qquad strongly\quad in \quad C([0,T];H_{loc}^{2-\alpha}).
\]
Therefore, combining \eqref{5.2}, we have
\begin{align}\label{5.3}
(v^{\tau},u^{\tau})\rightarrow(\Phi^0+\widetilde v^{0}, \Psi^0+\widetilde u^{0})
=:(v^0, u^0),
\qquad strongly\quad in \quad C([0,T];H_{loc}^{2-\alpha}).
\end{align}
On the other hand, noting that $\sqrt{\tau}\Pi^{\tau}$ is uniform bounded in $L^{\infty}((0,\infty);H^2)$, which yields $\tau \Pi_t ^{\tau}\rightarrow0$ in $D^{\prime}((0,\infty)\times \mathbb R)$ as $\tau\rightarrow0$.  Let $\tau\rightarrow0$ in \eqref{1.3},  we have
\begin{align}\label{5.4}
\Pi^{\tau}\rightharpoonup\mu\frac{(u^0)_x}{v^0}:=\Pi^0\qquad a.e. \quad (0,\infty)\times\mathbb{R}
\end{align}
and we conclude that $v^0,u^0$ are the solutions of classical Navier-Stokes equations.
Then, combining \eqref{5.1}, \eqref{5.2}, \eqref{5.3} and \eqref{5.4}, we get the desired results.

\appendix
\section{proof of Lemma \ref{le3.10}}

Firstly, we define $F_i(t, X_i)$ ($i=1, 2$) as follows:
\[
F_i(t, X_i)=-\frac{M}{\delta_i}
\left[\int_{\mathbb{R}}\frac{a}{\sigma_i}(\widetilde{u}_i^{X_i})_x(p(v)-p(\widetilde{v}))\dif x
-\int_{\mathbb{R}}a\left(p(\widetilde{v}_i^{X_i})\right)_x(v-\widetilde{v})\dif x\right].
\]
Applying \eqref{3.11-1}, Lemma \ref{le2.1} and the definition of $a(t, x)$ in \eqref{2.11}, we can obtain that
\[
\|a\|_{C^1}\le 2,\quad \|\widetilde{v}_i^{X_i}\|_{C^2}<\infty,\quad  \|(\widetilde{v}_i^{X_i})_x\|_{W^{1,1}}\le C\delta_i,\quad \|(\widetilde{u}_i^{X_i})_x\|_{W^{1,1}}\le C\delta_i.
\]
Then, we have
\begin{equation}\label{3.12-1}
\begin{aligned}
\sup\limits_{X_i\in\mathbb{R}}|F_i(t, X_i)|&\le \frac{C}{\delta_i}\|a\|_{C^1}\||p(\widetilde{v}_i)|,|p(v)|,|\widetilde{v}_i|,|v|\|_{L^\infty}
\left(\|(\widetilde{v}_i^{X_i})_x\|_{L^1}+\|(\widetilde{u}_i^{X_i})_x\|_{L^1}\right)\\
&\le C,
\end{aligned}
\end{equation}
and
\[
\begin{aligned}
\sup\limits_{X_i\in\mathbb{R}}|\partial_{X_i}F_i(t, X_i)|&\le \frac{C}{\delta_i}\|a\|_{C^1}\||p(\widetilde{v}_i)|,|p(v)|,|\widetilde{v}_i|,|v|\|_{L^\infty}
\left(\|(\widetilde{v}_i^{X_i})_{xx}\|_{L^1}+\|(\widetilde{u}_i^{X_i})_{xx}\|_{L^1}\right)\\
&\le C,
\end{aligned}
\]
where $C$ is a constant independent of $t$ and $\tau$. Therefore, the ODE \eqref{2.10} has a unique Lipschitz continuous solution by using the Cauchy-Lipschitz theorem (Lemma C.1 of \cite{KV9}).

In particularly, since $|\dot X_1(t)|+|\dot X_2(t)|\le C $ according \eqref{3.12-1}, we have \eqref{3.10-2}.

\section{proof of Lemma \ref{le4.3}}

Although the proof of Lemma \ref{le4.3} can be directly derived from the Appendix A of \cite{SMJ}, there are still slight differences between the traveling wave solutions of the relaxed system and classical system. Therefore, we still provide the proof of the Lemma \ref{le4.3} in this appendix.
We only estimates the case of $i=2$, other case $i=1$ can be followed in a similar way.
Firstly, according Lemma \ref{le2.1}, we can obtain that
\[
|(\widetilde{v}_2)^{X_2}_x|\le C\delta_2^2e^{-C\delta_2|x-\sigma_2 t-X_2(t)|}\quad \forall x\in\mathbb{R},\quad t>0,
\]
and
\[
|\widetilde v-(\widetilde{v}_2)^{X_2}|=|(\widetilde{v}_1)^{X_1}-v_m|\le
\begin{cases}
C\delta_1e^{-C\delta_1|x-\sigma_1 t-X_1(t)|}\quad \text{if} \quad x\ge\sigma_1t+X_1(t),\\
C\delta_1\quad\quad \quad\qquad \qquad  \qquad  \text{if} \quad x\le\sigma_1t+X_1(t).
\end{cases}
\]
Noting that $X_1(t)+\sigma_1t\le\frac{\sigma_1t}{2}\le0\le\frac{\sigma_2t}{2}\le X_2(t)+\sigma_2t$, we can derive that
\[
|(\widetilde{v}_2)^{X_2}_x||(\widetilde{v}_1)^{X_1}-v_m|\le
\begin{cases}
C\delta_1\delta_2^2e^{-C\delta_1|x-\sigma_1 t-X_1(t)|}\quad \text{if} \quad x\ge0,\\
C\delta_1\delta_2^2e^{-C\delta_2|x-\sigma_2 t-X_2(t)|} \quad \text{if}\quad x\le0.
\end{cases}
\]
On the other hand, we note that
\begin{align*}
&x-\sigma_1 t-X_1(t)\ge -\frac{\sigma_1t}{2}\ge0\quad \text{if} \quad x\ge0,\\
&x-\sigma_2 t-X_2(t)\le -\frac{\sigma_2t}{2}\le0\quad \text{if} \quad x\ge0.
\end{align*}
Therefore, we conclude that
\[
|(\widetilde{v}_2)^{X_2}_x||(\widetilde{v}_1)^{X_1}-v_m|\le
C\delta_1\delta_2^2e^{-C\min\{\delta_1, \delta_2\}t},\quad \forall x\in\mathbb{R},\quad t>0,
\]
and
\[
\int_{\mathbb{R}}|(\widetilde{v}_2)^{X_2}_x||(\widetilde{v}_1)^{X_1}-v_m|\dif x\le
C\delta_1\delta_2^2e^{-C\min\{\delta_1, \delta_2\}t},\quad t>0.
\]
In a similar way, using Lemma \ref{le2.1}, we can get
\[
|(\widetilde{v}_1)^{X_1}_x||(\widetilde{v}_2)^{X_2}_x|\le
C\delta_1^2\delta_2^2e^{-C\min\{\delta_1, \delta_2\}t},\quad \forall x\in\mathbb{R},\quad t>0,
\]
and
\[
\int_{\mathbb{R}}|(\widetilde{v}_1)^{X_1}_x||(\widetilde{v}_2)^{X_2}_x|\dif x\le
C\delta_1\delta_2e^{-C\min\{\delta_1, \delta_2\}t},\quad t>0.
\]

\textbf{Acknowledgement}: Yuxi Hu's Research is supported by the Fundamental Research Funds for the Central Universities (No. 2023ZKPYLX01) .

\textbf{Data Availibility}:  Data sharing not applicable to this article as no datasets were generated or analysed during the current study.

\textbf{Conflict of interest}: On behalf of all authors, the corresponding author states that there is no conflict of interest.

\end{document}